\documentclass[11pt,final]{amsart}

\usepackage[margin=1in]{geometry} 

\usepackage{amsmath}
\usepackage{amsthm}
\usepackage{amssymb}
\usepackage{float}
\usepackage{multirow}
\usepackage{calligra}
\usepackage{mathrsfs}
\usepackage[caption=false]{subfig}
\usepackage[utf8]{inputenc}
\usepackage[dvipsnames]{xcolor}
\usepackage[colorlinks=true, linkcolor=Purple, citecolor=Orange, urlcolor=Purple]{hyperref}

\usepackage[sort&compress,numbers,square]{natbib}
\usepackage{oubraces}
\usepackage{stix}
\usepackage{stmaryrd}
\theoremstyle{plain}
\newtheorem{theorem}{Theorem}[section]

\newtheorem{lemma}[theorem]{Lemma}

\newtheorem{assumption}[theorem]{Assumption}

\theoremstyle{definition}

\newtheorem{example}{Example}
\newtheorem{Remark}[theorem]{Remark}

\newcommand\bx{\boldsymbol{x}}

\newcommand\calP{{\mathcal{P}}}
\newcommand\calT{{\mathcal{T}}}

\usepackage{graphicx, color}
\graphicspath{{fig/}}

\usepackage{algorithm, algpseudocode} 
\usepackage{mathrsfs} 

\usepackage{lipsum}

\newcommand{\mesh}{\calP}
\newcommand{\meshsh}{\calP^{\sharp}}

\newcommand{\mean}[1]{\{\kern-1.1mm\{#1\}\kern-1.1mm\}}          
\newcommand{\jump}[1]{[\![#1]\!]}                        

\newcommand{\ud}{\,\mathrm{d}}
\newcommand{\ndg}[1]{| \kern -.25mm \|{#1}| \kern -.25mm \|}
\newcommand{\ndgp}[1]{| \kern -.25mm \|{#1}| \kern -.25mm \|_{\rm \mesh}}
\newcommand{\ndgps}[1]{| \kern -.25mm \|{#1}| \kern -.25mm \|_{\rm \mesh,s}}

\newcommand{\diam}{\operatorname{diam}}

\newcommand{\norm}[2]{\|{#1}\|_{{#2}}}
\newcommand{\ncdg}[1]{| \kern -.25mm \|{#1}| \kern -.25mm \|_{\rm DG}}

\renewcommand{\k}{K}
\newcommand{\ksh}{{K^{\sharp}}}

\newcommand{\bnabla}{\nabla_{\kern-.05cm\mesh}^{}}

\title[Composite Nodally bound-preserving DG methods on polytopic meshes]{A nodally bound-preserving  composite discontinuous Galerkin method  on polytopic meshes}

\author[A.~Amiri]{Abdolreza  Amiri}
\address{Department of Mathematics and
    Statistics, University of Strathclyde, 26 Richmond Street,
    Glasgow, G1 1XH United Kingdom}
                    \email{abdolreza.amiri@strath.ac.uk}

\author[G.R.~Barrenechea]{Gabriel R. Barrenechea}
\address{Department of Mathematics and
    Statistics, University of Strathclyde, 26 Richmond Street,
    Glasgow, G1 1XH United Kingdom}
\email{gabriel.barrenechea@strath.ac.uk}

\author[E.H.~Georgoulis]{Emmanuil H.~Georgoulis}
\address{  The Maxwell Institute for Mathematical Sciences \& Department of Mathematics, \\ Heriot-Watt University, Edinburgh EH14 4AS, UK, and \\ Department of Mathematics, School of Applied Mathematical and Physical Sciences, \\ National
  Technical University of Athens, Zografou 15780, Greece, and \\
  IACM-FORTH, Greece}
\email{e.georgoulis@hw.ac.uk}

\author[T.~Pryer]{Tristan Pryer}
\address{Department of Mathematical Sciences, University of Bath, Claverton down,  Bath BA2 7AY, UK}
\email{tmp38@bath.ac.uk}

\begin{document}
\maketitle

\begin{abstract}
 We introduce a nodally
bound-preserving Galerkin method for second-order elliptic
problems on general polygonal/polyhedral, henceforth collectively termed as \emph{polytopic}, meshes.  Starting
from an interior penalty discontinuous Galerkin (DG) formulation posed on a
polytopic mesh, the method enforces preservation of \emph{a priori} prescribed upper and lower bounds for the numerical solution at arbitrary number of user-defined points \emph{within} each polytopic element. This is achieved by employing a simplicial submesh and enforce bound preservation at the submesh nodes via a nonlinear iteration. By construction, the submeshing procedure preserves the order of accuracy of the DG method, \emph{without} introducing any additional global numerical degrees of freedom compared to the baseline DG method, thereby, falling into the category of composite finite element approaches. A salient feature of the proposed method is that it automatically reverts to the standard DG method on polytopic meshes when no prescribed bound violation occurs. In particular, the choice of the discontinuity-penalisation parameter is independent of the submesh granularity. The resulting composite method combines the geometric
flexibility of polytopic meshes with the accuracy and stability of
discontinuous Galerkin discretisations, while rigorously guaranteeing
bound preservation.  The existence and uniqueness of the numerical solution is proven. A priori error bounds, assuming sufficient regularity of the exact solution are shown, employing a non-standard construction of discrete nodally bound-preserving interpolant. Numerical experiments confirm optimal convergence for
smooth problems and demonstrate robustness in the presence of sharp gradients, such as boundary
and interior layers.
\end{abstract}

\section{Introduction}
\label{sec:intro}

High-order discontinuous Galerkin (DG) methods are attractive for
approximating partial differential equations because they combine
local conservation, geometric flexibility, and ease of
adaptivity. However, classical DG discretisations typically violate the
discrete maximum principle and may produce large overshoot or undershoots, resulting in non-physical approximations, e.g., not respect positivity properties of the exact solution. Such violation of the expected range/bounds of the exact solution by its numerical approximation may introduce stability challenges in the presence of nonlinearities or coupled systems.

For example, consider a model boundary-value problem involving a reaction--diffusion partial differential equation (PDE) on some computational domain $\Omega$, and suppose its exact solution satisfies the bounds
\[
   0 \leq u(x) \leq \kappa \quad \text{for almost every } x \in \Omega ,
\]
with $\kappa > 0$, under suitable assumptions, e.g., on the forcing term and
boundary conditions \cite{renardy2006introduction}. Numerical approximations $u_h$ of the exact solution $u$ by classical
finite element/discontinuous Galerkin methods may not be bound preserving, in
general: even if  $u(x)\in[0,\kappa]$, a.e.~$x \in \Omega$,  the set $\{x
\in \Omega : u_h(x) \notin [0,\kappa]\}$ may have positive
measure. Ciarlet and Raviart shown already in 1973  that the standard conforming finite element
method on simplicial meshes requires stringent mesh conditions to
guarantee a bound-preserving approximation \cite{CR73}.

Since negative densities and pressures in Euler's equations, or
probabilities, concentrations, and phase-field solutions lying outside
$[0,1]$, are physically meaningless, positivity-preserving, or more
generally bound-preserving, discretisations have received increasing
attention over the last few decades.  Approaches include
\cite{XZ99,BE05,DDS05,lipnikov2007monotone,Kuz07,EHS09,LMS11,
  BBK17-NumMath,vdVXX19,ChengShen22,HuanChen23,Kirby4755186} and the
recent monographs \cite{Kuzmin-book} for hyperbolic and
\cite{BJK-book} for elliptic second-order partial differential
equations. A common aspect of most of these methods is that achieving
bound preservation requires nonlinear approaches, either through
limiters or by reformulating the problem as a variational inequality.

Polytopic meshes provide several advantages over simplicial or
quadrilateral grids. Their ability to conform to complicated
geometrical features without excessive mesh granularity \cite{DG-EASE}, in conjunction with allowing local mesh coarsening via element agglomeration and, in general, their flexibility in dynamic mesh modification scenarios, make them attractive
for adaptive simulations \cite{cangiani2017hp, cangiani_VEM2017, cangianiVEMpar2021, cangiani2023}. Polytopic methods also admit efficient sweeping strategies for transport-type problems \cite{calloo2025cycle}. By
contrast, literature on positivity-preserving, or in general bound-preserving, discretisations for
polytopic meshes is scarce and usually restricted to low order. For
instance, \cite{ST22-VEM,SYG23-VEM} propose hybrid (nonlinear) finite
volume/Virtual Element Methods (VEMs) that are provably positivity-preserving in the low-order setting. In the finite volume
context, there is also a large body of work on positivity-preserving
low-order methods; see \cite{Droniou14} for a review. In the finite
element setting, the only FEM-based scheme we are aware of is
\cite{FRK20}, whereby a Flux-Corrected Transport (FCT) scheme based on a
DG formulation on polytopic meshes is introduced and shown to be
bound-preserving for element mean values (although numerical
experiments are carried out only on simplicial meshes).

The main goal of this work is to design and analyse an arbitrary order, provably
bound-preserving DG method for elliptic problems on polytopic
meshes. Given the allowed element shape generality, the method preserves predetermined bounds on \emph{user-defined} points within each element.  This is a crucial development in our view, as in many applications, positivity/bound preservation is  crucial within coupled multi-physics simulations where the
solution of a PDE, for instance a density, is passed as input into other model
components. In such settings, it is not necessary for the solution to
be bound-preserving everywhere; it is sufficient to ensure preservation only at the
\emph{quadrature points} used for the discretisation of other equations in the model.


The proposed approach starts from the polytopic interior penalty DG
method analysed in \cite{cangiani2014hp, cangiani2017hp} serving as the baseline discretisation. A crucial aspect for the success of polytopic DG in admitting elements of extreme shape generality is the careful choice of the discontinuity-penalisation parameter \cite{cangiani2014hp, cangiani2017hp}. Too small penalisation hinders stability. On the other hand, selecting the parameter much larger than needed for stability may restrict approximation quality, when no conforming subspace is available. Then, building on the nodally bound-preserving (NBP) conforming finite element method on simplicial meshes of \cite{BGPV24}, we consider the closed convex set of appropriate discrete functions satisfying the physical bounds at their degrees of
freedom, and define a nonlinear projection onto this set. A discrete problem is
then formulated for the projected object, and a stabilisation term is
added to remove the non-trivial kernel induced by the projection.  The kernel stabilisation is carefully constructed to satisfy a key monotonicity property (see Lemma \ref{l31} below), which allows for well-posedness of the discrete solution. This
strategy was first proposed for (linear
and nonlinear)  reaction--diffusion equations  \cite{BGPV24}, and has since been extended to convection--diffusion
equations with conforming finite elements \cite{ABP24}, DG methods
\cite{barrenechea2025nodally}, densities in proton therapy
\cite{ashby2025positivity}, and Allen--Cahn problems \cite{DEW25},
among others.

To ensure maximum control of the points where bound-preservation is enforced, the method employs simplicial submeshes of the polytopic elements and requires that bounds are preserved at the vertices of each submesh. The choice of submeshes can be freely selected, e.g., to coincide with quadrature nodes, for instance. By construction, the submeshing procedure preserves the order of accuracy of the DG method, \emph{without} introducing any additional global numerical degrees of freedom compared to the baseline polytopic DG method. A critical property of the proposed method is that the choice of the discontinuity-penalisation parameter remains subordinate to polytopic mesh only; in particular, it is independent of the submesh granularity. The aspect ratio between the ``coarse'' polytopic meshsize and the ``fine'' submesh size only appears in the definition of the kernel stabilisation term, and has no effect on the stability or approximation properties.  The resulting method is
regarded as \emph{composite} in the spirit of
\cite{HS97-Composite,HS97-Composite-II}, due to the presence of the submesh without proliferation of global degrees of freedom.

The well-posedness of the numerical solution is proven via the Brower-Minty Theorem, upon resorting to the key monotonicity property of the kernel stabilisation. Moreover, \emph{a priori} error bounds, assuming sufficient regularity of the exact solution are shown. In contrast to the case of NBP methods on simplicial meshes \cite{BGPV24, ABP24,DEW25}, whereby the standard interpolants are used to carry out the error analysis, the presence of polytopic meshes and the desire to enforce bound preservation at user-defined nodes poses a new technical challenge: standard 
(quasy-)interpolants are no longer applicable. To address this 
we employ a non-standard construction, inspired from \cite{mosco1974one,strang2006dimension,Kirby4755186}, which results in a NBP interpolant with optimal approximation properties. Optimal \emph{a priori} error bounds are, thus, proven for sufficiently smooth solutions. Moreover, we detail the practical construction of the method on the matrix level to showcase that it is possible to reuse existing implementations, thus preserving conservation and consistency properties of the original method. Finally, a comprehensive series of numerical experiments confirms the optimal convergence for
smooth problems, demonstrate robustness in the presence of sharp gradients, such as boundary
and interior layers, highlighting further the practicality of the proposed method.



To the best of our knowledge, the method presented below is the first rigorous, arbitrary order
positivity/bound-preserving Galerkin for polytopic
meshes.  We remark that the construction is new even when restricting to simplicial
meshes: submeshes with vertices positioned at the quadrature nodes of the baseline mesh can be used to able to enforce physical consistency in coupled problems  scenarios. The method is also related to the Recovered
Element methodology \cite{georgoulis2018recovered,dong2020recovered},
which also uses reconstruction and projection techniques to
enhance stability and approximation on general meshes.  


The remainder of this manuscript is organised as
follows. Section~\ref{sec:pre} presents the model problem, preliminary
results, and the baseline IPDG
discretisation. Section~\ref{Sec:NBP-poly} introduces the polygonal
NBP method and its well-posedness. Section~\ref{Sec:Interpol}
constructs an interpolation operator that respects the bounds in the
polytopic mesh, which is instrumental for the error analysis carried
out in Section~\ref{sec:Error}. Implementation details are given in
Section~\ref{Sec:Implementation}, followed by numerical experiments in
Section~\ref{sec:numerics} that demonstrate the performance of the
method and validate the theory. Conclusions are drawn in
Section~\ref{sec:conclusion}.

\section{Model problem and its discretisation by the discontinuous Galerkin method}
\label{sec:pre}	

We adopt standard notation for Sobolev spaces, following
\cite{EG21-I}. For a domain $D \subseteq \mathbb{R}^d$, we denote by
$\|\cdot\|_{0,p,D}$ the $L^p(D)$-norm, omitting the subscript $p$ when
$p=2$. For $s \geq 0$ and $p \in [1,\infty]$, we write
$\|\cdot\|_{s,p,D}$ (resp.\ $|\cdot|_{s,p,D}$) for the norm (resp.\
seminorm) in $W^{s,p}(D)$, and again drop the subscript $p$ when
$p=2$. We set
\[
H^1_0(D) = \{ v \in H^1(D) : v=0 \text{ on } \partial D \},
\]
and denote by $H^{-1}(D)$ the dual of $H^1_0(D)$, identifying $L^2(D)$
with its dual. Writing $\langle \cdot,\cdot \rangle_D$ for the duality
pairing, we have
\[
\langle f,v \rangle_D = \int_D f(\boldsymbol{x}) v(\boldsymbol{x})
 \mathrm{d}\boldsymbol{x}, \qquad v \in H^1_0(D),
\]
whenever $f \in H^{-1}(D)$ is sufficiently regular. In the sequel we
do not distinguish between inner products and duality pairings for
scalar- or vector-valued functions.

\subsection{Model problem} 

Let $\Omega \subset \mathbb{R}^d$ ($d=2,3$) be an open, bounded
polygonal or polyhedral domain with boundary $\partial\Omega$. For
given $f \in H^{-1}(\Omega)$, we consider the elliptic problem
\begin{equation}\label{CDR}
  \begin{aligned}
    - \operatorname{div}(\mathcal{D}\nabla u) + \mu u &= f && \text{in }\Omega, \\
    u &= 0 && \text{on }\partial\Omega,
  \end{aligned}
\end{equation}
where $\mu \in L^{\infty}(\Omega)$ with $\mu \ge 0$ a.e.\ in $\Omega$,
and $\mathcal{D} = (d_{ij})_{i,j=1}^d \in [L^{\infty}(\Omega)]^{d
  \times d}$ is symmetric and uniformly strictly positive
definite. That is, there exists a constant $\mathcal{D}_0 > 0$ such
that
\[
\boldsymbol{y}^\top \mathcal{D}(\boldsymbol{x})  \boldsymbol{y}  \ge  \mathcal{D}_0   \boldsymbol{y}^\top \boldsymbol{y}, 
\qquad \forall  \boldsymbol{y} \in \mathbb{R}^d \setminus \{0\}, \quad \text{for a.e.\ } \boldsymbol{x} \in \Omega.
\]
It follows from the Lax--Milgram lemma that problem \eqref{CDR} admits
a unique weak solution.  The analysis presented below extends to more
general settings, such as mixed boundary conditions and first-order
perturbation terms, but for clarity we restrict attention to the model
case \eqref{CDR}.

We seek discrete solutions that preserve the bounds satisfied by the
weak solution of \eqref{CDR}.  By the maximum and comparison
principles (see, e.g., \cite[Corollary~4.4]{renardy2006introduction}),
for almost all $\boldsymbol{x} \in \Omega$ the solution $u$ satisfies
\begin{equation}
-\|f\|_{0,\infty,\Omega} \mu^{-1}  \le  u(\boldsymbol{x})  \le  \|f\|_{0,\infty,\Omega} \mu^{-1}.
\end{equation}
If, in addition, $f \ge 0$ in $\Omega$, this bound sharpens to
\begin{equation}
0  \le  u(\boldsymbol{x})  \le  \|f\|_{0,\infty,\Omega} \mu^{-1}.
\end{equation}
Motivated by these bounds, we impose the following assumption on the
exact solution throughout this work.
\begin{assumption}[Bounded solution assumption]
The weak solution of \eqref{CDR} is assumed to satisfy
\begin{equation}\label{eq14}
0  \leq  u(\boldsymbol{x})  \leq  \kappa, 
\qquad \text{for almost all } \boldsymbol{x} \in \Omega,
\end{equation}
where $\kappa > 0$ is a given constant.
\end{assumption}

\subsection{Admissible partitions and finite element spaces} \label{sec2.2}

Let $\calP$ be a subdivision of $\Omega$ into disjoint polygonal
elements when $d=2$, or polyhedral elements when $d=3$. We will
collectively refer to these as polytopic elements. For a nonnegative
integer $k$, we denote the set of all polynomials of degree $k$ on $K
\in \calP$ by $\mathbb{P}_{k}(K)$. For $k \geq 1$, we consider the
elementwise discontinuous space
\begin{align}
  \label{eq:dgspace}
  V_\calP := \{ v_H \in L^{2}(\Omega) : v_H|_{K} \in \mathbb{P}_{k}(K) \ \ \forall K \in \calP \}.
\end{align}
We denote by $\Gamma_\calP := \bigcup_{K\in\calP}\partial K$ the
skeleton of $\calP$, and by $\Gamma_\calP^{\rm int} := \Gamma_\calP
\setminus \partial \Omega$ its interior part. The mesh function
$H_\calP:\bar{\Omega}\to\mathbb{R}_+$ is defined by
\[
H_\calP|_{\mathring{K}} := \mathrm{diam}(K) \qquad \forall K \in \calP.
\]

We assume that each element $K \in \calP$ can be subdivided into a
finite collection of simplices, denoted $\calT_K$, such that $K =
\bigcup_{T\in\calT_K} T$, and denote $\calT = \bigcup_{K\in \calP}
\calT_K$. No global conformity is required: the local triangulations
$\calT_K$ may be chosen independently. On each $K\in\calP$, we define
the local continuous piecewise polynomial space
\begin{align}
  W_{K} := \{ v_h \in C^{0}(K) : v_h|_{T} \in \mathbb{P}_{k}(T) \ \ \forall T\in \calT_{K} \},
\end{align}
and set $W_\calT := \bigoplus_{K\in\calP} W_{K}$.

A key feature of the proposed method is that the approximate solution
may be piecewise polynomial within each polytopic element $K \in
\calP$ in order to enforce bound preservation. Crucially, however, the
number of global degrees of freedom of the method coincides with $\dim
V_\calP$ rather than with the generally larger space $W_\calT$.
  
We now present assumptions on the admissible partitions and introduce some basic notation.

\begin{assumption}[Admissible polytopic meshes]\label{mesh_assumption}
  Let $({\mesh_i})_{i\in I}$ be a family of polytopic meshes. We assume:
  \begin{itemize}
  \item[(a)] Each element $K \in \mesh_i$ is star-shaped with respect
    to a ball of radius $\rho_K$, and there exists a global constant
    $C_{\rm star}>0$ such that
  \[
  H_K \le C_{\rm star} \rho_K \qquad \forall K \in \mesh_i, \ \forall i \in I.
  \]
\item[(b)] The family $({\mesh_i})_{i\in I}$ is locally quasi-uniform,
  i.e., there exists a constant $C_{\rm qu} \ge 1$ such that
  \[
  H_{K'} \le C_{\rm qu} H_K
  \]
  for all face neighbours $K'$ of each $K \in \mesh_i$, for all $i \in
  I$.
\end{itemize}
\end{assumption}

Note that Assumption~\ref{mesh_assumption}(b) implies the equivalence
\[
C_{\rm qu}^{-1} H_K  \le  H_{K'}  \le  C_{\rm qu} H_K
\]
for all face neighbours $K'$ of a polytopic element $K \in \mesh$.

The well-posedness and error analysis of the nodally bound-preserving
method also requires the following mild technical assumptions on the
polytopic partition $\calP$ and its sub-triangulation $\calT$.

\begin{assumption}[Admissible simplicial submeshes]\label{mesh_assumption_two}
  Let $({\calT_j})_{j\in J}$ be a family of simplicial submeshes
  obtained as refinements of a given polytopic mesh $\calP$. We
  assume:
\begin{itemize}
\item[(a)] The family $({\calT_j})_{j\in J}$ is shape-regular, i.e.,
  there exists a global constant $C_{\rm sh}>0$ such that
  \[
  h_T \le C_{\rm sh} \rho_T \qquad \forall T \in \calT_j, \ \forall j \in J,
  \]
  where $\rho_T$ is the inradius of $T$.
  \item[(b)] Let $K \in \calP$ and $T \in \calT_K$ be such that $T$
    has a facet $f \subset \partial K$. If $K' \in \calP$ is the
    neighbouring element with $f \subset K \cap K'$, then
  \[
  h_T \le H_{K'}.
  \]
\end{itemize}
\end{assumption}
Assumption~\ref{mesh_assumption_two}(a) is standard. Part (b) is a
mild technical condition requiring that the submesh of a polytopic
element is not coarser than its neighbouring polytopic elements.
 
We recall some standard results used throughout this work. Let $T
\subseteq \mathbb{R}^d$ be a simplex and $f$ one of its facets. The
following trace-inverse inequality holds (see \cite{war03}, and
\cite{DG-EASE} for extensions):
\begin{equation}\label{Trace-inverse}
  \|v_h\|_{0,f}^2  \le  \frac{(k+1)(k+d)|f|}{d|T|} \|v_h\|_{0,T}^2,
  \qquad \forall v_h \in \mathbb{P}_k(T).
\end{equation}
In addition, for every $v \in H^1(T)$ the local trace inequality holds
(see, e.g., \cite{EG21-I}):
\begin{equation}\label{local_trace_theorem}
  \|v\|_{0,f}^2
   \le  \|v\|_{0,T} |v|_{1,T} + \frac{|f|}{|T|} \|v\|_{0,T}^2.
\end{equation}
We also recall the standard inverse inequality (see, e.g.,
\cite{Sch98}): there exists a constant $C_{\rm inv}>0$, independent of
$T$, $k$, $v_h$, and $\rho_T$, such that
\begin{equation}\label{Standard-Inverse-Ineq}
  \|\nabla v_h\|_{0,T}^2  \le  C_{\rm inv} k^4 \rho_T^{-2} \|v_h\|_{0,T}^2,
  \qquad \forall v_h \in \mathbb{P}_k(T).
\end{equation}
For any element $K \in \calP$, we define its neighbourhood by
\[
  \omega_K := \{ K' \in \calP : K' \text{ shares a facet with } K \}.
\]
 
\subsection{Interior penalty discontinuous Galerkin method}

Let $K_+$ and $K_-$ be two adjacent elements of $\calP$ sharing an
interface $F \subset \partial K_+ \cap \partial K_- \subset
\Gamma_{\calP}^{\rm int}$.  For elementwise continuous scalar- and
vector-valued functions $v$ and $\mathbf{q}$, we define the
\emph{average} across $F$ by
\[
 \mean{v}|_F := \tfrac{1}{2}(v_+|_F + v_-|_F), 
 \qquad 
 \mean{\mathbf{q}}|_F := \tfrac{1}{2}(\mathbf{q}_+|_F + \mathbf{q}_-|_F),
\]
where $v_\pm|_F$ denotes the trace of $v$ from within $K_\pm$ on $F$,
and similarly for $\mathbf{q}$. The \emph{jump} across $F$ is defined
by
\[
 \jump{v}|_F := v_+ \mathbf{n}_+ + v_- \mathbf{n}_-,
 \qquad 
 \jump{\mathbf{q}}|_F := \mathbf{q}_+ \cdot \mathbf{n}_+ + \mathbf{q}_- \cdot \mathbf{n}_-,
\]
with $\mathbf{n}_\pm$ the outward unit normal to $K_\pm$ on $F$. On a
boundary face $F \subset \partial \Omega \cap \partial K$ we set
\[
 \mean{v} := v, 
 \qquad 
 \mean{\mathbf{q}} := \mathbf{q}, 
 \qquad 
 \jump{v} := v \mathbf{n}, 
 \qquad 
 \jump{\mathbf{q}} := \mathbf{q}\cdot \mathbf{n},
\]
with $\mathbf{n}$ the outward unit normal to $\partial \Omega$.

For brevity, we denote by $\nabla_{\calP} v$ the broken gradient of a
function $v:\Omega\to\mathbb{R}$ with $v|_K \in H^1(K)$ for all $K \in
\calP$, i.e., $(\nabla_{\calP} v)|_K = \nabla(v|_K)$.

The classical interior penalty discontinuous Galerkin (IPDG) method
is: find $u_H \in V_\calP$ such that
\begin{equation}\label{wip}
  a_{DG}(u_H,v_H) = \ell(v_H) 
  \qquad \forall v_H \in V_\calP,
\end{equation}
with $a_{DG}: (H^{3/2+\epsilon}(\Omega) + V_\calP)\times
(H^{3/2+\epsilon}(\Omega) + V_\calP)\to\mathbb{R}$, $\epsilon>0$, defined by
\begin{equation}\label{IPDG}
\begin{aligned}
a_{DG}(u_H,v_H) := & 
 \int_\Omega \big(\mathcal{D}\nabla_{\calP} u_H \cdot \nabla_{\calP} v_H + \mu u_H v_H\big)  \mathrm{d}\boldsymbol{x} 
 + \int_{\Gamma_\calP} \sigma_\calP \jump{u_H}\cdot \jump{v_H}  \mathrm{d}s \\
& - \int_{\Gamma_\calP} \big( \mean{\mathcal{D}\nabla u_H}\cdot \jump{v_H} 
  + \theta  \mean{\mathcal{D}\nabla v_H}\cdot \jump{u_H} \big)  \mathrm{d}s,
\end{aligned}
\end{equation}
and $\ell(v_H) := \langle f, v_H \rangle_{\Omega}$. Here $\theta \in
[-1,1]$ and $\sigma_\calP:\Gamma_\calP \to \mathbb{R}_{\ge 0}$ is the
\emph{penalty function}, whose precise form depends on the geometry of
the polytopic elements. The choice $\theta=1$ yields the symmetric
version of IPDG, while $\theta=-1$ gives its non-symmetric
counterpart.

The design of suitable penalty functions $\sigma_\calP$ for very
general polytopic or curved elements has been studied extensively; see
\cite{cangiani2014hp,cangiani2017hp,prismatic,DG-EASE,RIPDG}. In this
work, stronger assumptions on admissible meshes are required due to
the two-scale structure of the positivity-preserving method introduced
below.
 
We define the DG-norm by
\[
  \| v_H \|_{DG} := 
  \Big(
    \| \mathcal{D}^{1/2} \nabla_{\calP} v_H \|_{0,\Omega}^2
    + \| \sqrt{\mu}  v_H \|_{0,\Omega}^2
    + \| \sqrt{\sigma_\calP} \jump{v_H} \|_{0,\Gamma_{\calP}}^2
  \Big)^{1/2},
\]
for any $v_H \in H^{3/2+\epsilon}(\Omega) + V_\calP$, $\epsilon > 0$.  
With this norm, we obtain the following coercivity result for $a_{DG}(\cdot,\cdot)$.

\begin{lemma}[Coercivity of { $a_{DG}(\cdot,\cdot)$}]
  \label{lem:coercivity_coarse}
  Let $F\subset \partial\k_+\cap\partial \k_- $ be a generic face
  shared by two elements $\k_+,\k_-\in\mesh$.  When
  $F\subset\partial\Omega$, we set $\k_-=\emptyset$.  Defining
  \begin{equation}\label{penalty}
    \sigma_\mesh|_F:=8C_{\rm star}k(k-1+d)d^{-1}\max_{*\in\{+,-\}}\delta_{\k_*}H_{\k_*}^{-1},\quad\textrm{where}\quad \delta_{\k_*}:=\norm{\mathcal{D}}{0,\infty,K_*}^2\norm{\mathcal{D}^{-1}}{0,\infty,K_*}
  \end{equation}
  for every  face $F\subset\Gamma_{\mesh}$, we have 
  \begin{equation}\label{coercivity}
    a_{DG}(v,v)\ge \frac{3}{4} \ndgp{v}^2 \qquad\text{for all } v\in V_\mesh.
  \end{equation}
\end{lemma}

\begin{proof}
  For $v\in V_\mesh$, we first observe
  \begin{equation}\label{coer_basic}
    a_{DG}(v,v)\ge\ndgp{v}^2-2\norm{ \sigma_\mesh^{-1/2}\mean{\mathcal{D}\nabla v}}{0,\Gamma_{\mesh} }\norm{ \sqrt{\sigma_\mesh}\jump{v}}{0,\Gamma_{\mesh} }.
  \end{equation}  
  We focus on interior facets; boundary facets follow as a special
  case with $\k_-=\emptyset$. Each face $F\in\Gamma_{\mesh}$ is
  partitioned into simplices $\{f\}$ so that $\cup_{f\subset
    F}f=F$. For each simplicial $f$, define $K^f_*\subset K_*$,
  $*\in\{+,-\}$, as the simplex with face $f$ and opposite vertex the
  centre of the ball with respect to which $K_*$ is star-shaped.
  
  Applying the trace-inverse inequality \eqref{Trace-inverse} on $f$,
  we obtain
  \[
  \begin{aligned}
  \norm{ \sigma_\mesh^{-1/2}\mean{\mathcal{D}\nabla v}}{0,f}^2
  &\le
  \frac{1}{2\sigma_\mesh}
  \sum_{*\in\{+,-\}}\norm{\mathcal{D}|_{\k_*}}{0,\infty,f}^2\norm{ \nabla v|_{K_*}}{0,f}^2\\
  &\le
  \frac{1}{2\sigma_\mesh}
  \sum_{*\in\{+,-\}}\norm{\mathcal{D}|_{\k_*}}{0,\infty,f}^2\frac{k(k-1+d)|f|}{d|K^f_*|}\norm{ \nabla v|_{K_*}}{0,K_*^f}^2,
  \end{aligned}
  \]
  since $\nabla v\in [\mathbb{P}_{k-1}(K_*)]^d$ and $\sigma_\mesh$ is
  constant on each $F\in\Gamma_\mesh$ (hence also on each $f\subset
  F$).

  By construction, the simplices $K^f_*$ are disjoint and satisfy
  $\cup_{F\subset \partial\k} \cup_{f\subset F} K_*^f= K_*$.  Defining
  $H_f^{\perp}:=|K^f_*|/|f|$ as the height of $K^f_*$ with base $f$,
  we note that $H_f^{\perp}\ge \rho_K$. Thus,
  \[
  \frac{|f|}{|K^f_*|}\le  \rho_\k^{-1}\le C_{\rm star} H_{\k_*}^{-1}.
  \]
  Therefore,
  \[
  \begin{aligned}
  \norm{ \sigma_\mesh^{-1/2}\mean{\mathcal{D}\nabla v}}{0,\Gamma_{\calP}}^2
  &\le
  \frac{C_{\rm star}}{2}\sum_{F\subset \Gamma_{\mesh}}\sigma_\mesh^{-1}|_F\sum_{f\subset F}
  \sum_{*\in\{+,-\}}\norm{\mathcal{D}|_{\k_*}}{0,\infty,f}^2\norm{\mathcal{D}^{-1}}{0,\infty,\k_*}\frac{k(k-1+d)}{dH_{\k_*}}\norm{\mathcal{D}^{1/2} \nabla v}{0,\k_*^f}^2.
  \end{aligned}
  \]
  Selecting $\sigma_\mesh$ as in \eqref{penalty}, and recalling
  $\cup_{F\subset \partial\k_*} \cup_{f\subset F} K_*^f= K_*$, we
  obtain
  \begin{equation}\label{indef_est_important}
  \norm{ \sigma_\mesh^{-1/2}\mean{\mathcal{D}\nabla v}}{0,\Gamma_{\calP}}^2
  \le 
  16^{-1}\norm{\mathcal{D}^{1/2} \bnabla v}{0,\Omega}^2,
  \end{equation}
  which, combined with \eqref{coer_basic}, gives the result.
 \end{proof}
 
Note that the above (classical) IPDG method on polytopic meshes uses
the same number of elemental basis functions on each element,
regardless of its shape. This is achieved by defining the local basis
functions directly as restrictions of polynomials in the physical
space, rather than through element mappings. For details on the
implementation of this approach, see \cite{cangiani2017hp}.

\section{A nodally bound-preserving composite discontinuous Galerkin method}
\label{Sec:NBP-poly}

In this section we propose a method that enforces bound preservation
on nodal basis functions subordinate to the simplicial submesh $\calT$
defined above. We begin by noting that the cardinality of the
$k$th-order Lagrange basis (and nodes) in $d$ dimensions for a simplex
is $m_{k,d}={k+d \choose d}$. For each simplex $T\subset K$, with
$K\in\calP$, we denote the Lagrange basis functions and nodes by
$\{\phi_i^T\}_{i=1}^{m_{k,d}}$ and $\{\bx_i^T\}_{i=1}^{m_{k,d}}$,
respectively.

Defining $C(K)$ as the space of continuous functions on $K$, we
introduce the elemental nodally bound-preserving recovery operator
$\mathcal{E}_{K}^+: C(K)\to W_{K}^+$ by
\begin{equation}
  \mathcal{E}_{K}^+(v)
  =\sum_{T\in\calT_K}\sum_{i=1}^{m_{k,d}}
   \max\!\big\{0,\min\{v(\bx_i^T),\kappa\}\big\}\phi_i^T.
\end{equation}
By construction, $\mathcal{E}_{K}^+(v)\in W_{K}^+$, where
\begin{align}
  W_{K}^{+}
  :=\{v\in W_{K} : v(\bx_{i}^T)\in[0,\kappa],
   \ \ i=1,\ldots,m_{k,d}, \ T\in\calT_K\}. \label{eq16}
\end{align}
Thus, $\mathcal{E}_{K}^+(v)$ lies within the prescribed range
$[0,\kappa]$ at every nodal point $\bx_i^T$.  

The global nodally bound-preserving recovery operator
$\mathcal{E}^+: C(\Omega)\to W_{\calT}$ is then defined elementwise by
\[
  (\mathcal{E}^+(v))|_K:= \mathcal{E}^+_K(v|_K), \quad K\in\calP,
\]
and for notational convenience we set
\begin{align}
  \mathcal{E}^{-}_{K}(v_{H}):=v_{H}|_K-\mathcal{E}_{K}^{+}(v_{H}),
  \qquad
  \mathcal{E}^{-}(v_{H}):=v_{H}-\mathcal{E}^{+}(v_{H}).
  \label{eq188}
\end{align}

\begin{Remark}[Properties of the recovery operator]
The operator $\mathcal{E}^+$ acts as the identity whenever all nodal
values lie in $[0,\kappa]$. Specifically, if $v_H\in V_\calP$ satisfies
$v_H(\bx_i^T)\in[0,\kappa]$ for all $i=1,\ldots,m_{k,d}$ and
$T\in\calT_K$, then $\mathcal{E}_{K}^+(v_H)=v_H$ on $K$. Consequently,
if this condition holds for every $K\in\calP$, we have
$\mathcal{E}^+(v_H)=v_H$. Conversely, if at least one nodal value lies
outside $[0,\kappa]$, then $\mathcal{E}_{K}^+(v_H)\neq v_H$ on that
element. Finally, if $v_H(\bx_i^T)<0$ for all nodes of some
$K\in\calP$, then $\mathcal{E}_{K}^+(v_H)=0$, so the nonlinear map
$\mathcal{E}^+$ has a non-trivial kernel.
\end{Remark}

To mitigate the effect of the non-trivial kernel, we introduce a
\emph{stabilising} bilinear form. For $w_h,v_h\in W_{\calT}$ we set
\begin{equation}
  s(w_h,v_h)=\sum_{K\in\mesh}\alpha\sum_{T\in\calT_K}\sum_{i=1}^{m_{k,d}}
  \big(\mathcal{D}_{\omega_K} h_T^{d-2}+\mu_T h_T^{d} \big)
  w_h(\bx_i^T)v_h(\bx_i^T) ,\label{eq20}
\end{equation}
where $h_T:=\diam(T)$ and $\mu_T:=\|\mu\|_{0,\infty,T}$ for
$T\in\calT$, while
$\mathcal{D}_{\omega_K}:=\|\mathcal{D}\|_{0,\infty,\omega_K}$ for each
$K\in\calP$. Here $\alpha$ denotes a piecewise constant function with
$\alpha|_K:=\alpha_K>0$, $K\in\mesh$, to be specified below. The
bilinear form $s(\cdot,\cdot)$ induces the norm
$\|v_h\|_{s}:=\sqrt{s(v_h,v_h)}$ on $W_\calT$.

We are now in a position to introduce the composite
bound-preserving discontinuous Galerkin method, which reads: find
$u_H\in V_\calP$ such that
\begin{equation}
  a_{h}(u_{H};v_{H})=\ell(v_{H})
  \qquad \forall v_{H}\in V_\calP,\label{BP-DG}
\end{equation}
with semilinear form
\begin{align}
  a_{h}(u_{H};v_{H})
  :=a_{DG}(\mathcal{E}^{+}(u_{H}),v_{H})
    +s(\mathcal{E}^{-}(u_{H}),v_{H}) ,\label{FEMethod}
\end{align}

\begin{Remark}[Properties of the bound-preserving DG method]
Some observations on the method \eqref{BP-DG} are in order.
\begin{enumerate}
\item If $u_{H}(\bx_i^T)\in[0,\kappa]$ for all $i=1,\dots,m_{k,d}$ and
  $T\in\calT$, then \eqref{BP-DG} reduces to the classical interior
  penalty discontinuous Galerkin method \eqref{wip} with solution
  $u_H\in V_\calP$, i.e.\ the nonlinearity in the first argument of
  $a_h(\cdot;\cdot)$ vanishes.
\item The stabilisation term $s(\cdot,\cdot)$ is active only on
  polytopic elements $K$ where $u_H$ violates the prescribed bounds on
  the numerical solution. Moreover, whenever
  $0<u_{H}(\bx_i^T)<\kappa$, we have
  $\mathcal{E}^{-}(u_{H})(\bx_i^T)=0$; see \cite[Remark~3.2]{ABP24}
  for the proof in the conforming finite element case, which carries
  over directly here.
\item The system to be solved in \eqref{BP-DG} has size
  $\dim(V_\mesh)\times \dim(V_\mesh)$. Hence, the number of degrees of
  freedom is \emph{independent of the submesh $\calT$}.
\end{enumerate}
\end{Remark}

\subsection{Well-posedness}
We now discuss existence and uniqueness for \eqref{BP-DG}. We begin by
showing that $s(\cdot,\cdot)$ controls the kernel of the projection
$\mathcal{E}^+(\cdot)$.

\begin{lemma}
  \label{Lem:s}
There exists a constant $C_{\rm equiv}>0$, depending only on
$C_{\rm star}$, $C_{\rm sh}$, $k$, and the spatial dimension $d$, such
that for every $v_h\in W_\calT$,
\begin{align}
  \ndgp{ v_{h}}^{2}
   \le  C_{\rm equiv} \|\alpha^{-1/2} v_h \|_{s}^{2}.  \label{eq22}
\end{align}
\end{lemma}

\begin{proof}
Recall that the number of Lagrange basis functions (and nodes) of
degree $k$ in $d$ dimensions is $m_{k,d}=\binom{k+d}{d}$. For a simplex
$T\subset K$, $K\in\calP$, denote the Lagrange basis and nodes by
$\{\phi_i^T\}_{i=1}^{m_{k,d}}$ and $\{\bx_i^T\}_{i=1}^{m_{k,d}}$,
respectively. For an affine map $F_T:\hat T\to T$ from a reference
simplex $\hat T$, let $\{\psi_i\}_{i=1}^{m_{k,d}}$ be the corresponding
reference Lagrange basis so that $\phi_i^T\circ F_T=\psi_i$. Then
\[
\begin{aligned}
\|v_h\|_{0,T}^2
&\le m_{k,d}\sum_{i=1}^{m_{k,d}} v_h^2(\bx_i^T)\int_T \big(\phi_i^T(\bx)\big)^2 \mathrm{d}\bx\\
&\le m_{k,d} \max_{1\le i\le m_{k,d}}\|\phi_i^T\|_{0,\infty,T}^2 |T| 
   \sum_{i=1}^{m_{k,d}} v_h^2(\bx_i^T)
   =  C_{k,d}^{L^2} |T| \sum_{i=1}^{m_{k,d}} v_h^2(\bx_i^T),
\end{aligned}
\]
with $C_{k,d}^{L^2}:= m_{k,d} \max_{1\le i\le m_{k,d}}\|\psi_i\|_{0,\infty,\hat T}^2$,
since $\|\psi_i\|_{0,\infty,\hat T}=\|\phi_i^T\|_{0,\infty,T}$.

Using this bound together with the inverse inequality
\eqref{Standard-Inverse-Ineq} yields
\[
\|\mathcal{D}^{1/2}\nabla v_h\|_{0,T}^2
 \le  C_{\rm inv} k^4 \|\mathcal{D}\|_{0,\infty,T} \rho_T^{-2} \|v_h\|_{0,T}^2
 \le  C_{\rm inv} C_{k,d}^{L^2} k^4 \|\mathcal{D}\|_{0,\infty,T} \rho_T^{-2} |T|
        \sum_{i=1}^{m_{k,d}} v_h^2(\bx_i^T).
\]
Since $|T|\le h_T^d$ and, by Assumption~\ref{mesh_assumption_two}, the
shape-regularity gives $|T| \le C_{\rm sh}^2 \rho_T^2 h_T^{d-2}$, we conclude
\begin{equation}\label{H^1_stab}
  \|\mathcal{D}^{1/2}\nabla v_h\|_{0,T}^2
   \le  C_{k,d}^{H^1} \|\mathcal{D}\|_{0,\infty,T} h_T^{d-2}
          \sum_{i=1}^{m_{k,d}} v_h^2(\bx_i^T),
\end{equation}
with
\[
C_{k,d}^{H^1}:= C_{\rm inv} C_{\rm sh}^2 C_{k,d}^{L^2} k^4.
\]

Next, let $f\subset \partial T\cap\partial K$ be a face of the
simplicial subelement $T$ contained in the face $F$ of the polytopic
element $K$ containing $T$. Then
\[
\begin{aligned}
\|\sqrt{\sigma_\calP} v_h\|_{0,f}^2
&\le m_{k,d} \sigma_\calP \sum_{i=1}^{m_{k,d}} v_h^2(\bx_i^T)
    \int_f \big(\phi_i^T(\bx)\big)^2 \mathrm{d}s\\
&\le 4 C_{\rm star} m_{k,d} \max_{1\le i\le m_{k,d}}\|\psi_i\|_{0,\infty,\hat T} 
    \frac{k(k-1+d)}{d} 
    \|\mathcal{D}\|_{0,\infty,K_-\cup K_+} 
    \big(\min_{*\in\{-,+\}} H_{\k_*}\big)^{-1} |f| 
    \sum_{i=1}^{m_{k,d}} v_h^2(\bx_i^T),
\end{aligned}
\]
where $f\subset F\subset K_-\cap K_+$, and we have only retained the
nodes/basis functions that are nontrivial on $f$, which is a
$(d-1)$-simplex (hence the adapted numbering). By
Assumption~\ref{mesh_assumption_two}(b), $h_T\le \min_{*\in\{-,+\}}
H_{\k_*}$, and since $|f|\le h_T^{d-1}$ we obtain
\begin{equation}\label{pen_stab}
  \|\sqrt{\sigma_\calP} v_h\|_{0,f}^2
   \le 
  C_{k,d}^{\sigma} \|\mathcal{D}\|_{0,\infty,K_-\cup K_+} h_T^{d-2}
  \sum_{i=1}^{m_{k,d}} v_h^2(\bx_i^T),
\end{equation}
with
\[
C_{k,d}^{\sigma}:= \frac{4}{d} C_{\rm star} m_{k,d} k(k-1+d) 
                   \max_{1\le i\le m_{k,d}}\|\psi_i\|_{0,\infty,\hat T}.
\]

For implementation purposes we now bound each contribution in terms of
the stabilisation weights. Recalling the notation
$\mathcal{D}_{\omega_\k}$ and $\mu_T$ from the definition of
$s(\cdot,\cdot)$, combining the above estimates yields
\[
\ndgp{ v_{h}}^{2}
 \le 
C_{\rm equiv}\sum_{\k\in\mesh}\sum_{T\in\calT_\k}\sum_{i=1}^{m_{k,d}}
\big(\mathcal{D}_{\omega_\k} h_T^{d-2}+\mu_T h_T^{d}\big) v_h^2(\bx_i^T),
\]
with $C_{\rm equiv}:=\max\{C_{k,d}^{H^1}, 2 C_{k,d}^{\sigma}\}$ (note
that $C_{k,d}^{L^2}\le C_{k,d}^{H^1}$). The claim then follows from
the definition of $\|\cdot\|_{s}$.
\end{proof}

The next step towards proving well-posedness is the following
monotonicity result, whose proof is identical to that of
\cite[Lemma~3.1]{BGPV24} and is therefore omitted for brevity.

\begin{lemma}\label{l31}
The bilinear form $s(\cdot,\cdot)$, defined in \eqref{eq20}, satisfies
\begin{align}
  s(\mathcal{E}^{-}(v_{H})-\mathcal{E}^{-}(w_{H}),
    \mathcal{E}^{+}(v_{H})-\mathcal{E}^{+}(w_{H}))
  &\geq 0 \qquad \forall v_{H},w_{H} \in V_{\mathcal{P}}, \label{eq23}\\[0.5em]
  s(\mathcal{E}^{-}(v_{H}),w_{h}-\mathcal{E}^{+}(v_{H}))
  &\leq 0 \qquad \forall v_{H} \in V_{\mathcal{P}},\ 
     w_{h} \in W_\k^+,\ \k\in\mesh. \label{eq24}
\end{align}
\end{lemma}

We now establish coercivity and continuity of the stabilised
semilinear form $a_h(\cdot;\cdot)$ for a specific choice of trial and
test functions, which will be needed below. A crucial point in the
proof is that the discontinuity-penalisation parameter $\sigma_\calP$,
defined in \eqref{penalty}, is \emph{sufficient} to guarantee
stability even though the trial functions may be polynomials on the
submesh $\calT$.

\begin{lemma}\label{Lem:cont_twoscale}
Let $v_H,w_H\in V_\calP$, and set
$r_h^\pm:=\mathcal{E}^\pm(v_H)-\mathcal{E}^\pm(w_H)$ for brevity.
Define
\[
\calT_\k^\partial:=\{T\in \calT_\k:\ \exists \text{ face } f\subset \partial T\cap\partial K\}
\]
as the set of simplices $T$ in $\k$ touching the boundary of $\k$.  
Select $\alpha>0$ in \eqref{eq20} as
\begin{equation}\label{alpha_choice}
  \alpha|_\k:= \gamma \max\Big\{1,\frac{ H_{\k}}{8C_{\rm star}}\max_{T\in\calT_\k^\partial}\frac{|f|}{|T|}\Big\},\quad K\in\mesh,
\end{equation}
with $\gamma\ge 25C_{\rm equiv}$. Then
\begin{align}\label{cont_nonlinear}
a_{DG}(r_h^+, z_H) +s(r_h^-,z_H) 
\le \frac{7}{4}\Big(\ndgp{r_h^+}^2+\norm{r_h^-}{s}^2\Big)^{1/2} 
\Big(\ndgp{z_H}^2+\norm{z_H}{s}^2\Big)^{1/2}, 
\end{align}
for any $z_H\in V_\mesh$.
\end{lemma}

\begin{proof}
From standard estimates, we obtain
\begin{equation}\label{cont_one}
a_{DG}(r_h^+, z_H) \le \ndgp{r_h^+}\ndgp{z_H}
-	\int_{\Gamma_{\calP}}\mean{\mathcal{D}\nabla r_h^+}\cdot\jump{z_H}\ud s
-	\int_{\Gamma_{\calP}}\mean{\mathcal{D}\nabla z_H}\cdot\jump{r_h^+}\ud s.
\end{equation}

We now estimate the second and third terms on the right-hand side.  
For the last term, applying Cauchy--Schwarz together with
\eqref{indef_est_important} gives
\begin{equation}\label{indef_easy}
\Big|\int_{\Gamma_{\calP}}\mean{\mathcal{D}\nabla z_H}\cdot\jump{r_h^+}\ud s\Big|
\le 
4^{-1}\norm{\mathcal{D}^{1/2}\bnabla z_H}{0,\Omega}
\norm{\sqrt{\sigma_\calP}\jump{r_h^+}}{0,\Gamma_{\calP}}.
\end{equation}

For the remaining term, we decompose as
\begin{equation}\label{indef_cool}
\int_{\Gamma_{\calP}}\mean{\mathcal{D}\nabla r_h^+}\cdot\jump{z_H}\ud s
= \int_{\Gamma_{\calP}}\mean{\mathcal{D}\nabla (v_H-w_H)}\cdot\jump{z_H}\ud s
-\int_{\Gamma_{\calP}}\mean{\mathcal{D}\nabla r_h^-}\cdot\jump{z_H}\ud s
=:(I)-(II),
\end{equation}
since $r_h^++r_h^-=v_H-w_H$.

For $(I)$, Cauchy--Schwarz and \eqref{indef_est_important} yield
\[
\begin{aligned}
|(I)|
\le&\ 
4^{-1}\norm{\mathcal{D}^{1/2}\bnabla (v_H-w_H)}{0,\Omega}
\norm{\sqrt{\sigma_\calP}\jump{z_H}}{0,\Gamma_{\calP}}\\
\le&\
4^{-1}\big(\norm{\mathcal{D}^{1/2}\bnabla r_h^+}{0,\Omega}
+\norm{\mathcal{D}^{1/2}\bnabla r_h^-}{0,\Omega}\big)
\norm{\sqrt{\sigma_\calP}\jump{z_H}}{0,\Gamma_{\calP}}.
\end{aligned}
\]

For $(II)$, applying a trace--inverse inequality to every
$T\in\calT_\k^\partial$, $\k\in\mesh$, gives
\[
\begin{aligned}
|(II)|\le&\
\sum_{\k\in\mesh}\sum_{\substack{f\subset \partial T\cap\partial K\\ T\in \mathcal{T}_\k^\partial}}
\norm{\sigma_\calP^{-1/2}\mathcal{D}\nabla r_h^-|_T}{0,f}
\norm{\sqrt{\sigma_\calP}\jump{z_H}}{0,f}\\
\le&\
\bigg(\sum_{\k\in\mesh}\sum_{\substack{T\in \mathcal{T}_\k^\partial}}
\frac{\delta_{\k}|f|}{8C_{\rm star}|T|}
\min_{*\in\{+,-\}}H_{\k_*}\delta_{\k_*}^{-1}
\norm{\mathcal{D}^{1/2}\nabla r_h^-}{0,T}^2\bigg)^{1/2}
\norm{\sqrt{\sigma_\calP}\jump{z_H}}{0,\Gamma_{\calP}}.
\end{aligned}
\]

Since $\min_{*\in\{+,-\}}H_{\k_*}\delta_{\k_*}^{-1}\le H_\k\delta_{\k}^{-1}$ (as one of the $\k_*$ is $\k$ itself), and
\[
\alpha|_\k= \gamma \max\Big\{1,\frac{ H_{\k}}{8C_{\rm star}}\max_{T\in\calT_\k^\partial}\frac{|f|}{|T|}\Big\},
\]
we deduce
\begin{align}\label{cool_minus}
|(II)|
\le&\ \gamma^{-1/2}\bigg(\sum_{\k\in\mesh}\sum_{T\in \mathcal{T}_\k^\partial}
\norm{\sqrt{\alpha}\mathcal{D}^{1/2}\nabla r_h^-}{0,T}^2\bigg)^{1/2}
\norm{\sqrt{\sigma_\calP}\jump{z_H}}{0,\Gamma_{\calP}}\\
\le&\
\gamma^{-1/2}\norm{\sqrt{\alpha}\mathcal{D}^{1/2}\bnabla r_h^-}{0,\Omega}
\norm{\sqrt{\sigma_\calP}\jump{z_H}}{0,\Gamma_{\calP}}.\nonumber
\end{align}

Returning to \eqref{indef_cool}, the above bounds and Lemma
\ref{Lem:s}, with $r_h^-\in W_\calT$ and $\sqrt{\alpha}r_h^-\in
W_\calT$, imply
\[
\begin{aligned}
\Big|	\int_{\Gamma_{\calP}}\mean{\mathcal{D}\nabla r_h^+}\cdot\jump{z_H}\ud s\Big|
\le&\  4^{-1}\Big(\norm{\mathcal{D}^{1/2}\bnabla r_h^+}{0,\Omega}
+\norm{\mathcal{D}^{1/2}\bnabla r_h^-}{0,\Omega}
+4\gamma^{-1/2}\norm{\sqrt{\alpha}\mathcal{D}^{1/2}\bnabla r_h^-}{0,\Omega}\Big)
\\
&
\qquad \times\norm{\sqrt{\sigma_\calP}\jump{z_H}}{0,\Gamma_{\calP}}\\
\le&\  \tfrac{\sqrt{2}}{4}\Big(\norm{ \mathcal{D}^{1/2}\bnabla r_h^+}{0,\Omega}^2
+ 25 C_{\rm equiv}\gamma^{-1}\norm{ r_h^-}{s}^2\Big)^{1/2}
\norm{\sqrt{\sigma_\calP}\jump{z_H}}{0,\Gamma_{\calP}}.
\end{aligned}
\]

Since $\gamma\ge 25C_{\rm equiv}$, we conclude
\begin{equation}\label{indef_cool_too}
\Big|\int_{\Gamma_{\calP}}\mean{\mathcal{D}\nabla r_h^+}\cdot\jump{z_H}\ud s\Big|
\le 
\tfrac{\sqrt{2}}{4}\big(\norm{\mathcal{D}^{1/2}\bnabla r_h^+}{0,\Omega}^2
+\norm{r_h^-}{s}^2\big)^{1/2}
\norm{\sqrt{\sigma_\calP}\jump{z_H}}{0,\Gamma_{\calP}}.
\end{equation}

Finally, substituting \eqref{indef_easy} and \eqref{indef_cool_too}
into \eqref{cont_one}, and using $\sqrt{2}\le 2$, yields
\eqref{cont_nonlinear}.
\end{proof}

\begin{lemma}\label{Lem:mon_twoscale}
With the assumptions and notation of Lemma \ref{Lem:cont_twoscale}, we also have
\begin{equation}\label{mon_dec_lemma}
\begin{aligned}
a_{DG}(r_h^+,v_H-w_H)+s(r_h^-,v_H-w_H) 
 \ge  \frac{1}{2}\big(\ndgp{v_H-w_H}^2+\|r_h^-\|_{s}^2\big).
\end{aligned}
\end{equation}
\end{lemma}

\begin{proof}
Since $v_H-w_H=r_h^++r_h^-$, we obtain
\begin{equation}\label{mon_dec}
\begin{aligned}
a_{DG}(r_h^+,v_H-w_H)+s(r_h^-,v_H-w_H) 
 \ge  a_{DG}(v_H-w_H,v_H-w_H)-a_{DG}(r_h^-,v_H-w_H)+s(r_h^-,r_h^-) ,
\end{aligned}
\end{equation}
using $s(r_h^-,r_h^+)\ge 0$ from \eqref{eq23}.  
Since $v_H-w_H\in V_\mesh$, Lemma \ref{lem:coercivity_coarse} yields
\begin{equation}\label{mon_coarse_coer}
a_{DG}(v_H-w_H,v_H-w_H)\ge \frac{3}{4} \ndgp{v_H-w_H}^2.
\end{equation}

Next, arguing exactly as in the proofs of \eqref{indef_easy} and
\eqref{cool_minus}, we have
\[
\Big|\int_{\Gamma_{\calP}}\mean{\mathcal{D}\nabla( v_H-w_H)}\cdot\jump{r_h^-}  \mathrm{d}s\Big|
 \le  4^{-1} \|\mathcal{D}^{1/2}\bnabla ( v_H-w_H)\|_{0,\Omega} 
\|\sqrt{\sigma_\calP}\jump{r_h^-}\|_{0,\Gamma_{\calP}},
\]
and
\[
\Big|\int_{\Gamma_{\calP}}\mean{\mathcal{D}\nabla r_h^-}\cdot\jump{v_H-w_H}  \mathrm{d}s\Big|
 \le  \gamma^{-1/2}  \|\sqrt{\alpha} \mathcal{D}^{1/2}\bnabla r_h^-\|_{0,\Omega} 
\|\sqrt{\sigma_\calP}\jump{ v_H-w_H}\|_{0,\Gamma_{\calP}}.
\]
Applying Lemma \ref{Lem:s} to $\sqrt{\alpha} r_h^-\in W_\calT$ gives
\[
\gamma^{-1/2} \|\sqrt{\alpha} \mathcal{D}^{1/2}\bnabla r_h^-\|_{0,\Omega} \le  4^{-1} \|r_h^-\|_{s},
\]
for the stated choices of $\alpha$ and $\gamma$. Using these three
estimates together with the trivial bounds $\|\mathcal{D}^{1/2}\bnabla
v\|_{0,\Omega}\le \ndgp{v}$ and
$\|\sqrt{\sigma_\mesh}\jump{v}\|_{0,\Gamma_{\calP}}\le \ndgp{v}$ for
any $v\in W_\calT$, we obtain
\begin{equation}\label{mixed_bf}
|a_{DG}(r_h^-,v_H-w_H)|
 \le  \frac{5}{4} \ndgp{r_h^-} \ndgp{v_H-w_H}
+\frac{1}{4} \|r_h^-\|_{s} \ndgp{v_H-w_H}.
\end{equation}

Moreover, Lemma \ref{Lem:s} applied to $r_h^-\in W_\calT$ yields
\begin{equation}\label{twentyfive}
16 \ndgp{r_h^-}^2 \le  16  C_{\rm equiv} \|\alpha^{-1/2}r_h^-\|_{s}^2 \le  \|r_h^-\|_{s}^2,
\end{equation}
since $\alpha^{-1}\le \gamma^{-1}\le(16 C_{\rm equiv})^{-1}$ by
construction. Hence, from \eqref{mixed_bf} we conclude that
\[
|a_{DG}(r_h^-,v_H-w_H)|
 \le  \frac{21}{64} \|r_h^-\|_{s} \ndgp{v_H-w_H}
 <
 \frac{1}{4} \|r_h^-\|_{s}^2+\frac{1}{4} \ndgp{v_H-w_H}^2.
\]
Substituting this into \eqref{mon_dec}, together with
\eqref{mon_coarse_coer}, establishes \eqref{mon_dec_lemma}.
\end{proof}

The last two lemmas are sufficient to apply the theory of monotone
operators, thereby establishing existence and uniqueness of the
solution to \eqref{FEMethod}. In particular, we have the following
result.

\begin{theorem}[Well-posedness]
\label{Theorem11}
Let $T:V_\calP^{}\rightarrow [V_\calP^{}]'$ be defined by  
\[
[Tu_{H},v_{H}]:=a_h(u_H;v_H)=a_{DG}(\mathcal{E}^{+}(u_{H}),v_{H})+s(\mathcal{E}^{-}(u_{H}),v_{H}), \hspace{1cm} v_{H}\in V_\calP.
\]
Then, under the hypotheses of Lemma~\ref{Lem:cont_twoscale}, $T$ is
continuous and strongly monotone. Consequently, the problem
\eqref{FEMethod} has a unique solution $u_H\in V_\mesh$.
\end{theorem}

\begin{proof}  
To show that $T$ is continuous, we recall \eqref{twentyfive} and use
\eqref{cont_nonlinear}, which gives
\[
[Tv_H-Tw_H,v_H-w_H]\le \frac{7}{4}\Big(\ndgp{\mathcal{E}^+(v_H)-\mathcal{E}^+(w_H)}^2+\frac{26}{25}\norm{\mathcal{E}^-(v_H)-\mathcal{E}^-(w_H)}{s}^2\Big)^{1/2} \Big(\ndgp{z_H}^2+\norm{z_H}{s}^2\Big)^{1/2}.
\]
Thus, $T$ is continuous on $V_\mesh$.  

Moreover, \eqref{mon_dec_lemma} implies
\[
[Tv_H-Tw_H,v_H-w_H] \ge \frac{1}{2}\big(\ndgp{v_H-w_H}^2+\norm{\mathcal{E}^-(v_H)-\mathcal{E}^-(w_H)}{s}^2\big),
\]
showing that $T$ is strongly monotone on $V_\mesh$. Existence and
uniqueness therefore follow from the Browder-Minty Theorem (see,
e.g., \cite[Theorem~10.41]{renardy2006introduction}) together with
strong monotonicity.
\end{proof}

\begin{Remark}[Consistency of the method]  
Suppose $u$, the solution of \eqref{CDR}, belongs to
$H^{3/2+\epsilon}(\Omega)$, $\epsilon>0$, with $0\le u\le \kappa$ in
$\Omega$. Then \eqref{wip} is consistent in the sense that
\[
a_{DG}(u,v_H) =\ell (v_H),  
\]
for all $v_H\in V_\mesh$. Combining this identity with \eqref{BP-DG}
yields the following orthogonality relation for the bound-preserving
method \eqref{BP-DG}:
\begin{equation}\label{BP-consistency}
a_{DG}(u,v_H) =\ell (v_H)=a_{DG}(\mathcal{E}^+(u_H),v_H)+s(\mathcal{E}^-(u_H),v_H),
\end{equation}
for all $v_H\in V_\mesh$. 
\end{Remark}

\section{Bound-preserving best approximation}\label{Sec:Interpol}

The error analysis requires an optimal approximation of $u$ in
$V_\calP$ that is nodally bound-preserving at all nodes of
$\calT$. Since the degrees of freedom of $V_\calP$ are determined by
the polytopic elements, rather than by the submesh $\calT$, a
specialised construction is needed, which we detail in this
section. The construction relies on the following (mild) assumption on
the partitions.
    		
\begin{assumption}[Simplicial covering of polytopic meshes]
  \label{definitio36}
For each member $\mesh_i$ of a family of polytopic meshes
$(\calP_i)_{i\in I}$, we assume there exists a covering $\meshsh_i=
\{\ksh\}$, associated with the polytopic mesh $\calP$, consisting of
open, shape-regular $d$-simplices $\ksh$ of minimal possible diameter
$H_{\ksh}$, such that:
\begin{itemize}
\item[(a)] for each $\k \in\calP_i$, there exists a fixed $\ksh\in \meshsh_i$ such that $\k \subset \ksh$;
\item[(b)] there exists a global constant $C_{\rm sh}^{\sharp}\ge 1$ such that $H_{\ksh}\le C_{\rm sh}^{\sharp} H_\k$, for all $\k\in\mesh_i$, $i\in I$;
\item[(c)] each $\meshsh_i$ covers \emph{exactly} the domain $\Omega$, that is, $\Omega = \cup_{\ksh\in\meshsh_i}\ksh$.
\end{itemize}
\end{assumption}

The simplices $\ksh$ are allowed to overlap. From Definition
\ref{definitio36}(b) and Assumption \ref{mesh_assumption}(a), we
obtain
\[
|\ksh|\le H_{\ksh}^d\le \big(C_{\rm sh}^{\sharp} H_\k\big)^d\le \big(C_{\rm sh}^{\sharp} C_{\rm star}\big)^d \rho_\k^d\le \pi^{-1}\big(C_{\rm sh}^{\sharp} C_{\rm star}\big)^d |\k|,
\]
since $\rho_\k$ is the radius of a ball contained in $\k$. Setting
$C_{\rm cov}:=\pi^{-1}\big(C_{\rm sh}^{\sharp} C_{\rm star}\big)^d$,
we deduce
\begin{equation}\label{covering_bound}
 \sum_{\ksh\in\meshsh}|\ksh|\le C_{\rm cov}\sum_{\k\in\mesh}|K|=C_{\rm cov}|\Omega|,
\end{equation}
where the correspondence between $\k$ and $\ksh$ has been used.  

This bound shows that the assumptions on the mesh $\mesh$ and on the
simplicial covering $\meshsh$ are sufficient to control the extent of
overlap in the computational domain $\Omega$ caused by the use of
$\meshsh$.
	
We consider a nodally bound-preserving approximant in
$V_{\mesh}$. While related constructions appear in
\cite{mosco1974one,strang2006dimension,ashby2024duality}, our concrete
choice follows \cite{Kirby4755186}. The next theorem states its
optimal-order approximation.

\begin{theorem}[Bound-preserving best approximation]
  \label{theorembound12}
  Let $\k \in\calP$, let $k\in\mathbb{N}$ be the polynomial degree of
  $V_\mesh$, and let $v\in H^{k+1}(\ksh)$ be a function with range in
  $[0, \kappa]$.  Then, there exists $\pi_{H}v \in V_\calP$ whose
  range is contained in $[0, \kappa]$, such that
  \begin{align}
    |v - \pi_{H}v|_{m,\k} &\leq C_mH_{\k}^{k+1-m} |v|_{k+1,\ksh},\label{123456}
  \end{align}
  for $m=0,1,2$, where $\ksh\in\meshsh$ is as in Definition
  \ref{definitio36}. The constant $C_m>0$ depends on the
  shape-regularity of $\ksh$, on $r$, and on $k$. Moreover, for
  $m=1,2$, the constant $C_m>0$ also depends on $C_{\rm star}$, on the
  upper bound $\kappa$, and on $|v|_{1,d,\Omega}$ for $m=1$, or on
  $|v|_{2,\Omega}$ for $m=2$.
\end{theorem}

\begin{proof}
  Let $k\ge 2$ and let $i_{H}:H^2(\ksh)\to \mathbb{P}_k(\ksh)$ denote
  the Lagrange interpolation operator, producing the polynomial of
  total degree $k$ interpolating $v$ at the Lagrange nodes of the
  simplex $\ksh$. If, in addition, $v\in H^{k+1}(\ksh)$, then the
  standard best-approximation estimate holds:
  \begin{equation}
    \norm{v- i_{H}v}{0,\ksh}+H_{\ksh} |v-i_{H}v|_{1,\ksh}+H_{\ksh}^2 |v-i_{H}v|_{2,\ksh}+|\ksh|^{1/2}\norm{v- i_{H}v}{0,\infty,\ksh}\le  \tilde{C}_{\rm app} H_{\ksh}^{k +1}|v |_{k +1,\ksh}, 
    \label{lagranges}
  \end{equation} 
  for each $\k\in\mesh$, with $\tilde{C}_{\rm app}>0$ independent of
  $H_{\ksh}$ and $v$; see \cite[Theorem 3.1.4]{ciarlet}. Using the
  properties in Definition \ref{definitio36}, this yields
  \begin{equation}\label{lagrange}
    \norm{v- i_{H}v}{0,\k}+H_{\k} |v-i_{H}v|_{1,\k}+H_{\k}^2 |v-i_{H}v|_{2,\k}+|\k|^{1/2}\norm{v- i_{H}v}{0,\infty,\k}\leq  C_{\rm app} H_{\k}^{k+1}|v |_{k+1,\ksh},
  \end{equation} 
  with $C_{\rm app}:= \tilde{C}_{\rm app} \big(C_{\rm
    sh}^{\sharp}\big)^{k+1}$.
  
  By construction, the nodal values of $i_{H}v$ on $\ksh$ lie in
  $[0,\kappa]$ since $v$ itself does, but the full range of $i_{H}v$
  is not necessarily contained in $[0,\kappa]$. Consequently,
  $i_{H}v|_\k$ does not, in general, preserve bounds.
  
  To obtain an elementwise polynomial approximant of $v$ that is
  bound-preserving, we follow the construction in \cite[Theorem
    2]{Kirby4755186} for simplicial meshes. For each $\k\in \mesh$,
  define
  \begin{align*}
    (\pi_{H}v)|_\k :=\frac{\kappa}{2} + \beta_\kappa(v) \left(i_{H}v - \frac{\kappa}{2}\right), \quad\text{with}\quad \beta_\kappa(v):=\kappa\big(\kappa+2\norm{v -i_{H}v}{0,\infty,K}\big)^{-1}.
  \end{align*}
  Since $\pi_{H}v $ is obtained by linear operations and $V_\mesh$
  contains constants, it follows that $ \pi_{H}v \in V_\mesh$. To
  prove that its range is contained in $[0, \kappa]$, note first that
  any function $w$ with range in $[0,\kappa]$ satisfies
\begin{align}
\left\|w -\frac{\kappa}{2}\right\|_{0,\infty,K} \leq  \frac{\kappa}{2}.\label{inequalit32456}
\end{align}
Therefore,
\[
\left\|\pi_{H}v  - \frac{\kappa}{2}\right\|_{0,\infty,\k} = \beta_\kappa(v) \left\|i_{H}v - \frac{\kappa}{2}\right\|_{0,\infty,\k} \le \frac{\kappa}{2},
\]
since
\[
\left\|i_{H}v - \frac{\kappa}{2}\right\|_{0,\infty,\k}\le \norm{i_{H}v -v}{0,\infty,\k}+\norm{v - \frac{\kappa}{2}}{0,\infty,\k}\le \norm{i_{H}v -v}{0,\infty,\k}+\frac{\kappa}{2}.
\]

Thus $0\le \pi_{H}v\le \kappa$. We now estimate its approximation properties. From
\begin{equation}\label{interp_id}
	v - 	\pi_{H}v = \Big( v - \frac{\kappa}{2} \Big)(1-\beta_\kappa(v))+\beta_\kappa(v) ( v -i_{H}v ),
\end{equation}
the triangle inequality, \eqref{inequalit32456}, and \eqref{lagrange} yield
\begin{equation}\label{bound345}
	\norm{v - \pi_{H}v}{0,K} 
	\leq \sqrt{|\k|}\|v -i_{H}v\|_{0,\infty,K} + \|v - i_{H}v\|_{0,K}
	\leq  2C_{\rm app} H_\k^{r+1}|v|_{r+1,\ksh}.
\end{equation}

For the $H^1$-seminorm estimate, differentiating \eqref{interp_id},
taking norms, and applying the triangle inequality gives
\begin{align}
	\norm{\nabla (v - 	\pi_{H}v)}{0,\k} \le (1-\beta_\kappa(v))\norm{\nabla v}{0,\k}+\beta_\kappa(v) \norm{\nabla( v -i_{H}v )}{0,\k}. 
\end{align}
By H\"older’s inequality, $ \norm{\nabla v}{0,\k}\le |\k|^{(d-2)/(2d)}
\norm{\nabla v}{0,d, \k}, $ for $d=2,3$. Since $H^{k+1}(\ksh)$, $k\ge
2$, embeds continuously into $W^{1,d}(\ksh)$, the above combined with
\eqref{lagrange} gives
\[
\norm{\nabla (v - 	\pi_{H}v)}{0,\k} 
\le C_{\rm app}\big(2\kappa^{-1}|\k|^{-1/d}H_\k\norm{\nabla v}{0,d, \k} +1\big) H_\k^{k}|v|_{k+1,\ksh}.
\]
From Assumption \ref{mesh_assumption}(a), $ |\k|^{-1/d}H_\k\le C_{\rm
  star}^{-1}, $ so the result follows with
\[
C_1:=C_{\rm app}(2C_{\rm star}^{-1}\kappa^{-1}\norm{\nabla v}{0,d, \k}
+1).
\]
For the $H^2$-seminorm estimate, a similar argument yields
\[
|v - 	\pi_{H}v|_{2,\k} 
\le C_{\rm app}\big(2\kappa^{-1}|\k|^{-1/2}H_\k^2| v|_{2,\k} +1\big) H_\k^{k-1}|v|_{k+1,\ksh}.
\]
From Assumption \ref{mesh_assumption}(a), $ |\k|^{-1/2}H_\k^2\le
C_{\rm star}^{-d/2}H_\k^{2-d/2}, $ and the result follows with
\[
C_2:=C_{\rm app}( 2C_{\rm star}^{-d/2}\kappa^{-1}|v|_{2,
  \k}H_\k^{2-d/2} +1).\]

This completes the proof for $k\ge 2$. For $k=1$, we can set
$\pi_Hv=i_Hv$, since the range of linear Lagrange basis functions over
a simplex is contained in $[0,1]$. Standard interpolation estimates
apply in this case, with $C_2=1$.
\end{proof}

\section{\emph{A priori} error analysis}\label{sec:Error}

We are now in position to establish \emph{a priori} error bounds between
sufficiently smooth exact solutions $u$ of \eqref{CDR} and their nodally
bound-preserving approximations $\mathcal{E}^+(u_H)$ obtained through
\eqref{BP-DG}. 

\begin{theorem}[A priori error estimate]
  \label{thm:a_priori}
Let $u\in H^1_0(\Omega)$ be the solution to \eqref{CDR}, and let
$\mathcal{E}^+(u_H)\in W_\calT$ denote the nodally bound-preserving
approximation produced by \eqref{BP-DG}, with $\sigma_\mesh$ given by
\eqref{penalty} and $\alpha$ by \eqref{alpha_choice}, computed on
families of polytopic meshes and corresponding simplicial subdivisions
satisfying Assumptions \ref{mesh_assumption},
\ref{mesh_assumption_two}, and \ref{definitio36}. If, in addition,
$u\in H^{k+1}(\Omega)\cap H^1_0(\Omega)$, then the following error bound
holds:
\begin{equation}\label{a_priori_bound}
		\ndgp{u-\mathcal{E}^+(u_H)}\le  C_{\rm apr}\bigg(\sum_{\k\in\mesh}\big(\delta_{\omega_\k}+ H_\k^2\mu_\k\big)
		H_{\k}^{2k} |u|_{k+1,\ksh}^2\bigg)^{1/2},
\end{equation} 
where $\mu_\k:=\|\mu\|_{0,\infty,\k}$, with $\{\ksh\}$ denoting the
simplicial covering from Assumption~\ref{definitio36}, and where
$C_{\rm apr}>0$ depends only on the polynomial degree
$k\in\mathbb{N}$ and on the constants appearing in Assumptions
\ref{mesh_assumption}, \ref{mesh_assumption_two}, and \ref{definitio36}.
\end{theorem}

\begin{proof}
We begin by decomposing the error $e:=u-\mathcal{E}^+(u_H)$ as
\[
e=\big(u-\pi_H u\big) +\big( \pi_H u - \mathcal{E}^+(u_H)\big) =: \eta+\xi,
\]
noting that $\xi\in W_\calT$. The triangle inequality gives
$\ndgp{e}\le \ndgp{\eta}+\ndgp{\xi}$.

We now estimate $\xi\in W_\calT$. Observe that
$\xi = \mathcal{E}^+(\pi_H u)-\mathcal{E}^+(u_H)$, since
$0\le \pi_H u\le \kappa$.  
Applying Lemma~\ref{Lem:mon_twoscale}, we obtain
\begin{equation}\label{coer_xi}
  \begin{aligned}
    a_{DG}(\xi,\pi_H u-u_H)-s(\mathcal{E}^-(u_H),\pi_H u-u_H) \ge \tfrac{1}{2}\big(\ndgp{\pi_H u-u_H}^2+\norm{\mathcal{E}^-(u_H)}{s}^2\big).
  \end{aligned}
\end{equation}

Next, the modified consistency relation \eqref{BP-consistency} implies
\begin{equation}\label{GO2}
	a_{DG}(\xi,v_H) -s(\mathcal{E}^-(u_H),v_H)=-a_{DG}(\eta,v_H),
\end{equation}
and inserting this into \eqref{coer_xi} yields
\begin{equation}\label{coercivity_erroran}
\frac{1}{2}\big(\ndgp{\pi_H u-u_H}^2+\norm{\mathcal{E}^-(u_H)}{s}^2\big)\le  -a_{DG}(\eta,\pi_H u-u_H). 
\end{equation}

For brevity, set $\xi_H:=\pi_H u-u_H$. We now estimate the right-hand
side of \eqref{coercivity_erroran}. A straightforward estimate gives
\begin{equation}\label{cont_DG_error}
  |a_{DG}(\eta,\xi_H)|\le \ndgp{\eta}\ndgp{\xi_H}+\int_{\Gamma_\mesh}|\mean{\mathcal{D}\nabla \eta}| |\jump{\xi_H}|\ud s +\int_{\Gamma_\mesh}|\mean{\mathcal{D}\nabla \xi_H}| |\jump{\eta}|\ud s. 
\end{equation}
We treat the last two terms of \eqref{cont_DG_error} separately. For
the last term we proceed as in \eqref{indef_easy} to obtain
\[
\int_{\Gamma_\mesh}|\mean{\mathcal{D}\nabla \xi_H}| |\jump{\eta}|\ud s \le \tfrac{1}{4}\norm{\mathcal{D}^{1/2}\bnabla \xi_H}{0,\Omega}\norm{\sqrt{\sigma_\mesh}\jump{\eta}}{0,\Gamma_\mesh}.
\]

To estimate the term involving $\eta$, we proceed as in the proof of
Lemma~\ref{lem:coercivity_coarse}. More precisely, every face $F$ shared
by two elements $\k_+,\k_-\in\mesh$ is divided into simplices $\{ f\}$.
For each $f\subset F$ we recall the simplex $\k_*^f\subset \k_*$,
$*\in\{+,-\}$, with face $f$ and opposite vertex the centre of the ball
with respect to which $\k_*$ is star-shaped. Application of the local
trace theorem \eqref{local_trace_theorem} in each $\k_*^f$ yields
\begin{align}
	\sum_{F\subset \partial K_*}\norm{v}{0,F}^2
&\le\ C_{\rm tr} \sum_{F\subset \partial K_*}\sum_{f\subset F}\big(H_{\k_*}^{-1} \norm{v}{0,\k_*^f}^2+ H_{\k_*} | v|_{1,\k_*^f}^2 \big)\nonumber\\
&\le\ C_{\rm tr} \big(H_{\k_*}^{-1} \norm{v}{0,\k_*}^2+ H_{\k_*} | v|_{1,\k_*}^2 \big),\label{trace_theorem}
\end{align}
with $C_{\rm tr}$ depending only on the star-shapedness constant
$C_{\rm star}$ from Assumption~\ref{mesh_assumption}(a). We have also
used that
$\cup_{F\subset \partial\k_*}\cup_{f\subset F}\k_*^f=\k_*$ by
construction (see, e.g., \cite[Lemma 4.7]{DG-EASE} for a proof of the
trace theorem with explicit $C_{\rm tr}$).

Hence, \eqref{trace_theorem} implies
\[
\norm{\sqrt{\sigma_\mesh}\jump{\eta}}{0,\Gamma_\mesh}^2
\le 
\tilde{c}_1\sum_{\k\in\mesh} \delta_{\omega_\k}\big(
H_{\k}^{-2}\norm{\eta}{0,\k}^2+|\eta|_{1,\k}^2\big), 
\]
with  $\tilde{c}_1:= 8 C_{\rm tr}C_{\rm star}k(k-1+d)d^{-1}$.
Theorem~\ref{theorembound12} then gives
\begin{equation}\label{pen_one_error}
\begin{aligned}
	\norm{\sqrt{\sigma_\mesh}\jump{\eta}}{0,\Gamma_\mesh}^2
	\le &\ 
	c_1\sum_{\k\in\mesh} \delta_{\omega_\k}H_{\k}^{2k} |v|_{k+1,\ksh}^2  
	,
\end{aligned}
\end{equation}
with $c_1=\tilde{c}_1\max\{C_0^2,C_1^2\}$.

Now consider the second term on the right-hand side of
\eqref{cont_DG_error}, for which
\[
\int_{\Gamma_\mesh}|\mean{\mathcal{D}\nabla \eta}| |\jump{\xi_H}|\ud s \le \norm{\sigma_\mesh^{-1/2}\mean{\mathcal{D}\nabla \eta}}{0,\Gamma_\mesh}\norm{\sqrt{\sigma_\mesh}\jump{\xi_H}}{0,\Gamma_\mesh}.
\]
We estimate the term involving $\eta$ as follows:
\[
\begin{aligned}
	\norm{\sigma_\mesh^{-1/2}\mean{\mathcal{D}\nabla \eta}}{0,\Gamma_\mesh}^2 
	\le &\ 
\tilde{c}_2	\sum_{F\subset \Gamma_\mesh}\sum_{*\in\{+,-\}}\delta_{\k_*}^{-1}\norm{\mathcal{D}}{0,\infty,\k_*}^2\big( |\eta|_{1,\k_*^F}^2+H_{\k_*}^2 |\eta|_{2,\k_*^F}^2\big)\\
	\le &\ 
\tilde{c}_2	\sum_{\k\in\mesh}\norm{\mathcal{D}^{-1}}{0,\infty,\k}^{-1}\big( |\eta|_{1,\k}^2+H_{\k}^2 |\eta|_{2,\k}^2\big),
\end{aligned}
\]
for $\tilde{c}_2:=\big(8C_{\rm star}k(k-1+d)\big)^{-1}d^2 C_{\rm
  tr}(1+\max_{\k\in\mesh}H_\k)$, using the trace theorem
\eqref{trace_theorem} with $v=|\nabla \eta|_{\k_*}|$, the definition
of $\delta_{\k_*}$, and collecting contributions to each $\k_*$. Since
$\norm{\mathcal{D}^{-1}}{0,\infty,\k}^{-1}\le
\norm{\mathcal{D}}{0,\infty,\k}$, application of
Theorem~\ref{theorembound12} yields
\begin{equation}\label{pen_two_error}
\begin{aligned}
	\norm{\sigma_\mesh^{-1/2}\mean{\mathcal{D}\nabla \eta}}{0,\Gamma_\mesh}^2 
	\le &\ 
	c_2	\sum_{\k\in\mesh}\norm{\mathcal{D}}{0,\infty,\k}H_{\k}^{2k} |u|_{k+1,\ksh}^2,  
\end{aligned}
\end{equation}
with $c_2:=\tilde{c}_2\max\{C_1^2,C_2^2\}$.

Finally, using \eqref{pen_one_error} and standard estimates also yields
\begin{equation}\label{P_norm_error}
	\ndgp{\eta}^2\le \sum_{\k\in\mesh}\big(C_1^2 \norm{\mathcal{D}}{0,\infty,\k}+C_0^2 H_\k^2\|\mu\|_{0,\infty,\k}+c_1\delta_{\omega_\k}\big)H_{\k}^{2k} |u|_{k+1,\ksh}^2.
\end{equation}
Using \eqref{pen_one_error}, \eqref{pen_two_error}, and
\eqref{P_norm_error} to further estimate \eqref{cont_DG_error}, together
with the discrete Cauchy--Schwarz inequality, gives
\begin{equation}\label{error_xi_H}
  \begin{aligned}
    |a_{DG}(\eta,\xi_H)|
    \le &\
    2\big(
    \ndgp{\eta}^2
    +	\norm{\sigma_\mesh^{-1/2}\mean{\mathcal{D}\nabla \eta}}{0,\Gamma_\mesh}^2  
    +  \norm{\sqrt{\sigma_\mesh}\jump{\eta}}{0,\Gamma_\mesh}^2
    \big)^{1/2} \ndgp{\xi_H}\\
    \le&\ c_3\bigg(\sum_{\k\in\mesh}\big(\delta_{\omega_\k}+ H_\k^2\mu_\k\big)
    H_{\k}^{2k} |u|_{k+1,\ksh}^2\bigg)^{1/2} \ndgp{\xi_H},
  \end{aligned}
\end{equation}
with $c_3:=2(\max\{2c_1+c_2+C_1^2, C_0^2\})^{1/2}$ and
$\mu_\k:=\|\mu\|_{0,\infty,\k}$.

Returning now to \eqref{coercivity_erroran}, we use \eqref{error_xi_H}
(along with the inequality $ab\le a^2+b^2/4$ on the right-hand side of
\eqref{error_xi_H}) to deduce
\begin{equation}\label{first_bound_xi_H}
\frac{1}{4}\ndgp{\xi_H}^2+\frac{1}{2}\norm{\mathcal{E}^-(u_H)}{s}^2\le  c_3^2\sum_{\k\in\mesh}\big(\delta_{\omega_\k}+ H_\k^2\mu_\k\big)
H_{\k}^{2k} |u|_{k+1,\ksh}^2.
\end{equation}

From this estimate, we can obtain an upper bound for $\ndgp{\xi}$. Since
$\xi_H=\xi-\mathcal{E}^-(u_H)$, the triangle inequality and
\eqref{twentyfive} give 
\begin{equation}\label{young}
	\begin{aligned}
		\ndgp{\xi_H}^2\ge&\  \big(\ndgp{\xi}-\ndg{\mathcal{E}^-(u_H)}\big)^2\ge  \big(\ndgp{\xi}-5^{-1}\norm{\mathcal{E}^-(u_H)}{s}\big)^2\\
		\ge&\  \ndgp{\xi}^2 +\tfrac{1}{25}\norm{\mathcal{E}^-(u_H)}{s}^2 -\tfrac{2}{5} \ndgp{\xi}\norm{\mathcal{E}^-(u_H)}{s}
		\ge \tfrac{4}{5}\ndgp{\xi}^2 -\tfrac{4}{25}\norm{\mathcal{E}^-(u_H)}{s}^2, 
	\end{aligned}
\end{equation}
using $2ab\le a^2/5+5b^2$ in the last step. Combining
\eqref{first_bound_xi_H} with \eqref{young} gives
\begin{equation}\label{second_bound_xi}
	\ndgp{\xi}^2\le  5c_3^2\sum_{\k\in\mesh}\big(\delta_{\omega_\k}+ H_\k^2\mu_\k\big)
	H_{\k}^{2k} |u|_{k+1,\ksh}^2,
\end{equation}
where we have ignored the non-negative term
$\norm{\mathcal{E}^-(u_H)}{s}^2$. A final application of the triangle
inequality completes the proof.
\end{proof}

\begin{Remark}[Further consequences of the \emph{a priori} analysis]
  \phantom{Some remarks}
  \begin{enumerate}
  \item The appearance of the covering simplices $\ksh$ on the
    right-hand side of \eqref{a_priori_bound} does not affect the
    inferred convergence order. Indeed, estimating
    \eqref{a_priori_bound} from above and using \eqref{covering_bound}
    yields
    \[
    \ndgp{u-\mathcal{E}^+(u_H)}\le  C \big(\max_{\k\in\mesh} H_\k^k\big) |u|_{k+1,\Omega},
    \]
    where $C>0$ depends on $C_{\rm apr}$, $C_{\rm cov}$,
    $\mathcal{D}$, and $\mu$ only.
    
  \item The proof of Theorem \ref{thm:a_priori} also provides an
    \emph{a priori} error bound for the `non-compliant' approximation
    $u_H\in V_\mesh$, which may not be nodally
    bound-preserving. Specifically, from \eqref{first_bound_xi_H} we
    obtain a bound on $\ndgp{\pi_Hu-u_H}$, which, combined with
    \eqref{P_norm_error}, leads to
    \[
    \ndgp{u-u_H}\le  \tilde{C}_{\rm apr}\bigg(\sum_{\k\in\mesh}\big(\delta_{\omega_\k}+ H_\k^2\mu_\k\big)
    H_{\k}^{2k} |v|_{k+1,\ksh}^2\bigg)^{1/2},
    \] 
    for some $\tilde{C}_{\rm apr}>0$ with the same dependence on
    constants as $C_{\rm apr}>0$.
    
  \item The bound \eqref{first_bound_xi_H} also sheds light on the
    decay rate of $\mathcal{E}^-(u_H)\to 0$, i.e., the `non-compliant'
    component of the approximation, as $\max_{\k\in\mesh}H_\k\to 0$
    and/or $\max_{T\in\calT}h_T\to 0$.
    
    For simplicity, assume $\calT$ is quasi-uniform with
    representative submesh diameter $h$ and that $\mesh$ is
    quasi-uniform with representative element diameter $H$. From
    \eqref{alpha_choice}, each elemental component is proportional to
    $H h^{-1}$.
		
    Consider first the case $H$ fixed and $h\to 0$. Then the
    right-hand side of \eqref{first_bound_xi_H} is bounded by $C_{\rm
      up}H^k$, for some $C_{\rm up}>0$ depending on $u$. Using the
    definition of the stabilisation norm and $\alpha\sim H h^{-1}$, we
    obtain
    \[
    \begin{aligned}
      C_{\rm up}H^{2k}\ge \norm{\mathcal{E}^-(u_H)}{s}^2\ge&\ C_{\rm down}H h^{-3}\sum_{\k\in\mesh}\sum_{T\in\calT_\k}\sum_{i=1}^{m_{k,d}} h^{d}\big(\mathcal{E}^-(u_H)(\bx_i^T)\big)^2\sim C_{\rm down}H h^{-3} \|\mathcal{E}^-(u_H)\|_{0,\Omega}^2,
    \end{aligned}
    \]
    for $C_{\rm down}>0$ depending on $\mathcal{D}_0$. This shows that
    $\|\mathcal{E}^-(u_H)\|_{0,\Omega}\sim h^{3/2}$ for fixed
    $H$. Moreover, since $h^{-3}\ge H^{-3}$, we find
    $\|\mathcal{E}^-(u_H)\|_{0,\Omega}\sim H^{k+1}$, which is of
    higher order than what one might expect from
    \eqref{first_bound_xi_H}.
  \end{enumerate}
\end{Remark}

%

\section{Implementation and matrix structure}\label{Sec:Implementation}

The composite method \eqref{BP-DG} can be realised by a transformation
starting from the discontinuous Galerkin (DG) formulation posed on the
simplicial submesh $\calT$. We first introduce the discontinuous
polynomial space
\begin{align*}
  V_\calT := \bigl\{ v_{h}\in L^{2}(\Omega): v_{h}|_{T}\in \mathbb{P}_{k}(T)
  \quad \forall\, T\in \calT \bigr\}.
\end{align*}
On this space, the classical interior-penalty DG method reads: find
$u_{h}\in V_\calT$ such that
\begin{align}
  a_{\mathrm{DG}}^{\calT}(u_{h},v_{h})
  =
  \ell(v_{h}) \qquad \forall\, v_{h}\in V_\calT, \label{DG121}
\end{align}
where the bilinear form $a_{\mathrm{DG}}^{\calT}:
(H^{3/2+\epsilon}_0(\Omega)+V_\calT)\times(H^{3/2+\epsilon}_0(\Omega)+V_\calT)\to\mathbb{R}$ is
\begin{equation}
  \begin{split}
    a_{\mathrm{DG}}^{\calT}(u_{h},v_{h})
    &=
    \int_{\Omega}\!\Big(\mathcal{D} \nabla_{\calT} u_{h}\cdot\nabla_{\calT} v_{h}
    + \mu u_{h}v_{h}\Big)\,\mathrm{d}\boldsymbol{x}
    \\
    &\quad
    +
    \int_{\Gamma_\calT}\!\sigma_\mesh \jump{u_h}\cdot \jump{v_h}\,\mathrm{d}s
    -
    \int_{\Gamma_{\calT}}\!\Big(\mean{\mathcal{D}\nabla u_h}\cdot \jump{v_h}
    + \theta \,\mean{\mathcal{D}\nabla v_h}\cdot \jump{u_h}\Big)\,\mathrm{d}s, \label{DG11}
  \end{split}
\end{equation}
and $\ell(v_h):=\langle f,v_h\rangle_{\Omega}$. Here, $\theta\in[-1,1]$
is the usual IPDG parameter, $\sigma_\mesh$ the penalty parameter
(cf.~\eqref{penalty}), and $\Gamma_\calT=\cup_{T\in\calT}\partial T$ the
skeleton of $\calT$.

Next, recall the conforming space $W_\calT$ defined in
Section~\ref{sec2.2}. On this space we seek $w_{h}\in W_\calT$ such
that
\begin{equation}
  a_{\#}(w_{h},v_{h}) = \ell(v_{h}) \qquad \forall\, v_{h}\in W_\calT, \label{DG22} 
\end{equation}
with bilinear form
\begin{equation}
  \begin{split}
    a_{\#}(w_{h},v_{h})
    &=
    \int_\Omega \Big(\mathcal{D}\nabla_{\calT} w_{h}\cdot \nabla_{\calT} v_h
    + \mu w_{h}v_h\Big)\,\mathrm{d}x 
    +
    \int_{\Gamma_\mesh}\!\sigma_\mesh \jump{w_{h}}\cdot \jump{v_h}\,\mathrm{d}s
    \\
    &\quad
    -
    \int_{\Gamma_{\mesh}}\!\Big(\mean{\mathcal{D}\nabla w_{h}}\cdot \jump{v_h}
    + \theta \,\mean{\mathcal{D}\nabla v_h}\cdot \jump{w_{h}}\Big)\,\mathrm{d}s.
  \end{split}
\end{equation}

The spaces $V_\calT$ and $W_\calT$ are related through their local
bases. Specifically, each basis function $\phi_{i}^{K}$ of $W_{K}$,
for $K\in \calP$, can be expressed as a linear combination of basis
functions $\varphi_{l}^{K}$ of $V_{\calT}$:
\begin{align}
	\phi_{i}^{K}=\sum_{l=1}^{N_{K}^{i}}  \varphi_{l}^{K}, \qquad i=1,\ldots,N_{K},
\end{align}
where $N_{K}$ is the number of degrees of freedom of $W_{K}$,
$\mathbf{x}_{i}^{K}$ the node associated with $\phi_{i}^{K}$,
$N^{i}_{K}$ the number of coincident degrees of freedom in $V_\calT$,
and $\varphi_{l}^{K}$ the corresponding basis functions.

From this relation it follows that $a_{\#}$ can be written in terms of
$a_{\mathrm{DG}}^{\calT}$:
\begin{align}
  a_{\#}(\phi_{i}^{K_{1}},\phi_{j}^{K_{2}})
  =
  \sum_{k=1}^{N^{j}_{K_{2}}}\sum_{l=1}^{N^{i}_{K_{1}}}a_{\mathrm{DG}}^{\calT}(\varphi_{l}^{K_{1}},\varphi_{k}^{K_{2}}), 
  \qquad K_{1},K_{2}\in \calP. \label{Projection1}
\end{align}
The same holds for penalty contributions: if
$\Gamma_\mesh\cap\Gamma_\calT=\varnothing$ the jumps vanish, giving
\begin{align}
  \int_{\Gamma_\mesh} \sigma_\mesh\jump{\phi_{i}^{K_{1}}}\cdot\jump{\phi_{j}^{K_{2}}}\,\mathrm{d}s
  =
  \sum_{k=1}^{N^{j}_{K_{2}}}\sum_{l=1}^{N^{i}_{K_{1}}}\int_{\Gamma_{\calT}}\sigma_\mesh \jump{\varphi_{l}^{K_{1}}}\cdot\jump{\varphi_{k}^{K_{2}}}\,\mathrm{d}s. \label{penalty11}
\end{align}

Let $N_{\calT} := \dim(V_\calT)$ and $N_{\#} := \dim(W_\calT)$. Then
\eqref{Projection1} shows that $a_{\mathrm{DG}}^{\calT}$ and $a_{\#}$
are connected by a transformation matrix
$\mathbf{O}\in\mathbb{R}^{N_{\#}\times N_{\calT}}$:
\begin{align}
  \mathbf{A}_{\#}
  =
  \mathbf{O}\,\mathbf{A}_{\mathrm{DG}}\,\mathbf{O}^{T}, \label{Transformation}
\end{align}
where $[\mathbf{A}_{\#}]_{ij}=a_{\#}(\phi_{i},\phi_{j})$ is the
stiffness matrix for \eqref{DG22}, and
$[\mathbf{A}_{\mathrm{DG}}]_{ij}=a_{\mathrm{DG}}^{\calT}(\varphi_{i},\varphi_{j})$
the stiffness matrix for \eqref{DG121}. The algebraic problem
\eqref{DG22} can therefore be written: find
$\mathbf{W}\in \mathbb{R}^{N_{\#}}$ such that
\begin{align}
  \mathbf{A}_{\#}\mathbf{W}=\mathbf{O}\,\mathbf{A}_{\mathrm{DG}}\,\mathbf{O}^{T}\mathbf{W}= \mathbf{O}\mathbf{b},
\end{align}
where $\mathbf{b}\in \mathbb{R}^{N_{\calT}}$ has entries
$b_{i}=\langle f,\varphi_{i}\rangle_{\Omega}$.

Let $V_{\calP}$ be the global dG space defined in
(\ref{eq:dgspace}). Its basis functions $\psi_{j}^{K}\in V_{\calP}$
can be written in terms of $\{\phi_{i}^{K}\}$:
\begin{align}
  \psi_{j}^{K}=\sum_{i=1}^{N_{K}}a_{i}^{K}\phi_{i}^{K}. \label{Projection11}
\end{align}
This yields
\begin{align}
  a_{\rm DG}(\psi_{i}^{K_{1}},\psi_{j}^{K_{2}})
  = \sum_{l=1}^{N_{K_{2}}}\sum_{k=1}^{N_{K_{1}}}a_{l}^{K_{1}}a_{k}^{K_{2}}
  a_{\#}(\phi_{l}^{K_{1}},\phi_{k}^{K_{2}}), 
  \qquad K_{1},K_{2}\in \calP. \label{Transformation23}
\end{align}
Introducing
$\mathbf{Q}\in\mathbb{R}^{N_{\calP}\times N_{\#}}$ as the matrix
mapping $W_\calT$ to $V_\calP$, the global stiffness matrix is
\begin{align}
  \mathbf{A}_{\calP}=\mathbf{Q}\,\mathbf{A}_{\#}\,\mathbf{Q}^{T}. \label{Transformation2}
\end{align}
The block diagonal structure of $\mathbf{Q}$ is shown in
Figure~\ref{K}. Combining \eqref{Transformation} and
\eqref{Transformation2}, the algebraic system for the positivity-preserving R-FEM is
\begin{align}
  \mathbf{A}_{\calP}(\mathbf{Q}\mathbf{W}^{+}) + \mathbf{S}(\mathbf{Q} \mathbf{W}^{-})
  = \bigl(\mathbf{Q}\mathbf{O}\mathbf{A}_{\rm DG}\mathbf{O}^{T}\mathbf{Q}^{T}\bigr)(\mathbf{Q}\mathbf{W}^{+})
  + \mathbf{S}(\mathbf{Q}\mathbf{W}^{-})
  =  \mathbf{F}_\calP, \label{R-FEM_Algebric}
\end{align}
with $\mathbf{F}_\calP= \mathbf{Q}\mathbf{O}\mathbf{b}$. The
stabilisation matrix $\mathbf{S}$ is defined via $\calT$ as
\begin{align*}
  s_{\calT}(v_h,u_h)=\sum_{\k\in\mesh}\alpha\sum_{T\in\calT_\k}\sum_{i=1}^{m_{k,d}}
  \big(\mathcal{D}_{\omega_\k} h_T^{d-2}+\mu_T h_T^{d} \big)v_{h}(\boldsymbol{x}_{i})u_{h}(\boldsymbol{x}_{i}), 
  \qquad u_{h},v_{h}\in V_{\calT},
\end{align*}
with $\alpha$ from \eqref{alpha_choice}. Its matrix form is
\begin{align*}
  \mathbf{S}(\mathbf{Q} \mathbf{W}^{-})=
  \bigl(\mathbf{O}\mathbf{Q}\mathbf{S}_{\calT}\mathbf{Q}^{T}\mathbf{O}^{T}\bigr)(\mathbf{Q}\mathbf{W}^{-}).
\end{align*}

Figure~\ref{meshs} shows a representative polytopic mesh together with
two associated triangular submeshes obtained by connecting quadrature
nodes within each polytope: one with 198 triangles (baseline
quadrature) and a finer one with 1626 triangles. Figure
\ref{Fig4Transportation} shows the structure of the transformation
meshes over this representative mesh with Figure \ref{FigSMAT} showing
the structure of the stiffness matrices.

\begin{figure}[h!]
  \centering
  \subfloat[Polytopic mesh with 32 elements.]{\label{M3}
    \includegraphics[width=0.3\textwidth]{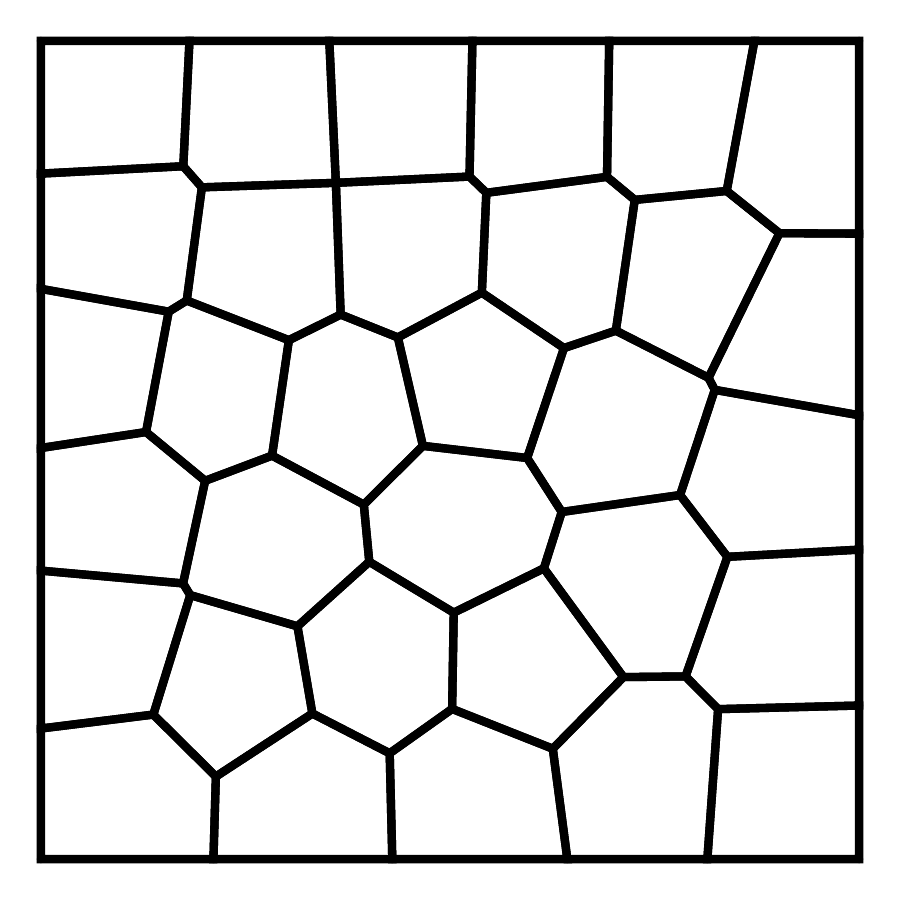}}
  \hfill
  \subfloat[Triangular submesh of Mesh~\ref{M3} with 198 elements.]{\label{M3triangulation}
    \includegraphics[width=0.3\textwidth]{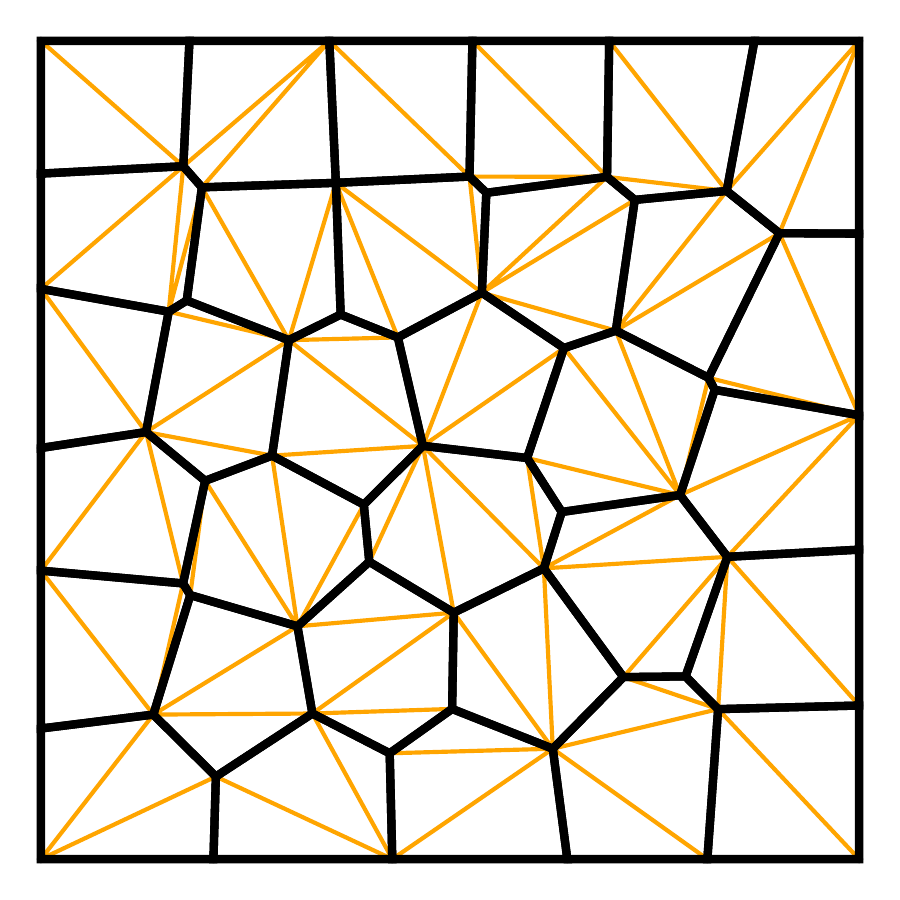}}
  \hfill
  \subfloat[Finer triangular submesh of Mesh~\ref{M3} with 1626 elements.]{\label{M3triangulationFine}
    \includegraphics[width=0.3\textwidth]{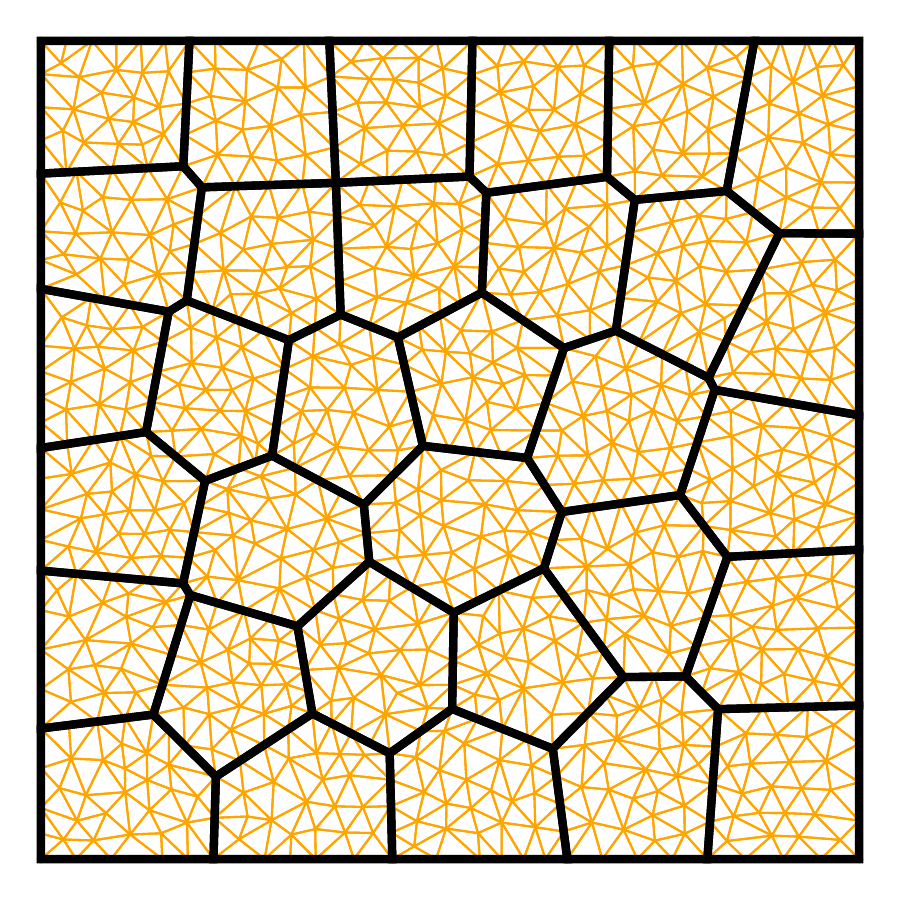}}
  \caption{Representative polytopic mesh and two associated triangular submeshes at different quadrature levels.}
  \label{meshs}
\end{figure}

\begin{figure}[h!]
  \centering
  \subfloat[$\mathbf{O}\in \mathbb{R}^{216\times  594}$]{\label{O} \includegraphics[width=0.5\textwidth]{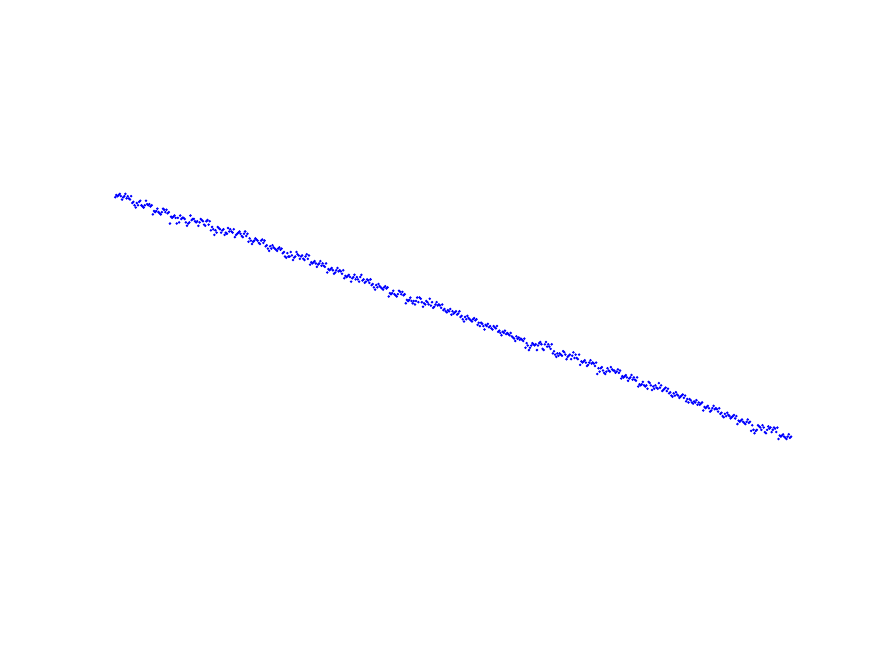}}
  \subfloat[$\mathbf{Q} \in  \mathbb{R}^{96\times216}$]{\label{K} \includegraphics[width=0.5\textwidth]{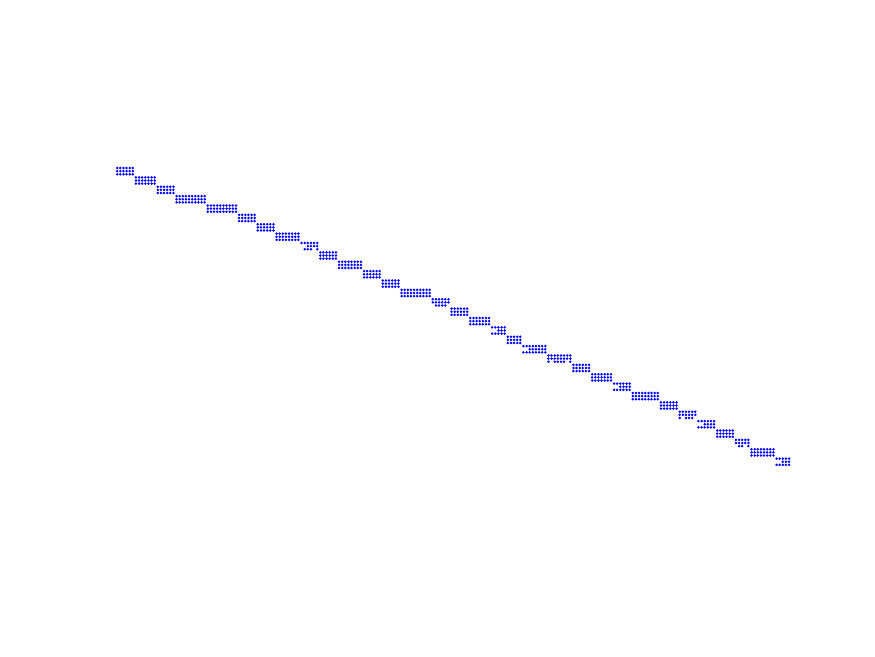}}
  \caption{Structure of the transformation matrices $\mathbf{O}$ and
    $\mathbf{Q}$ on Mesh~\ref{M3} with $\mathbb{P}_{1}$ elements.}
  \label{Fig4Transportation}
\end{figure}

\begin{figure}[h!]
  \centering
  \subfloat[$\mathbf{A}_{\rm DG}\in \mathbb{R}^{594\times594}$]{ \includegraphics[width=0.35\textwidth]{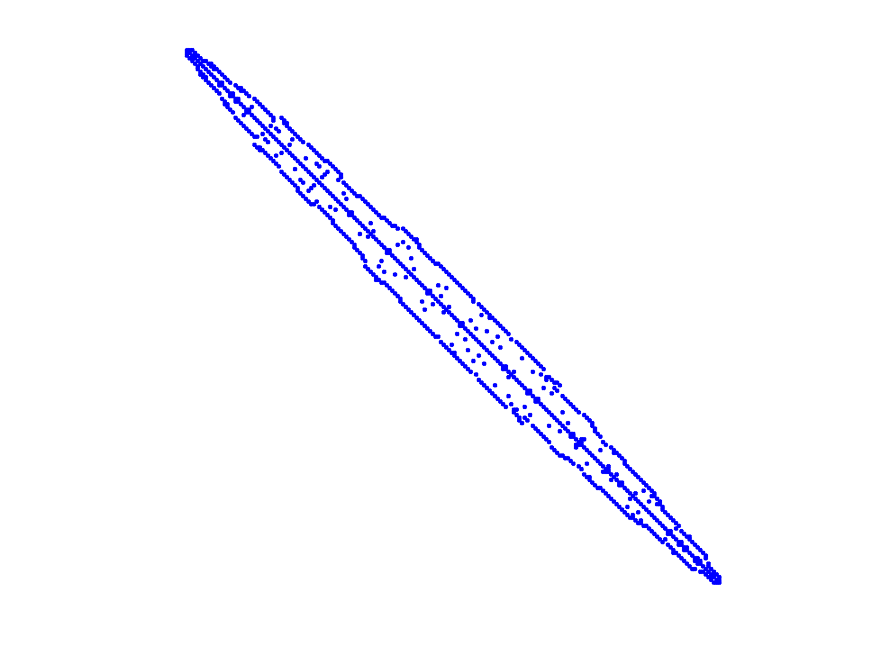}}
  \subfloat[$\mathbf{A}_{\#}\in \mathbb{R}^{216 \times216}$]{ \includegraphics[width=0.35\textwidth]{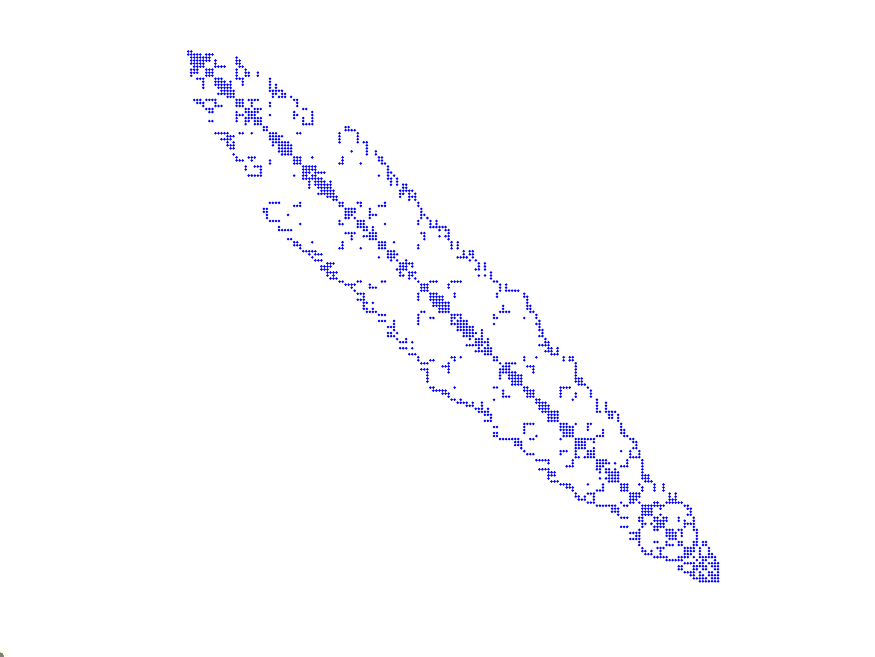}}
  \subfloat[$\mathbf{A}_{\calP}\in \mathbb{R}^{96\times96}$ ]{ \includegraphics[width=0.35\textwidth]{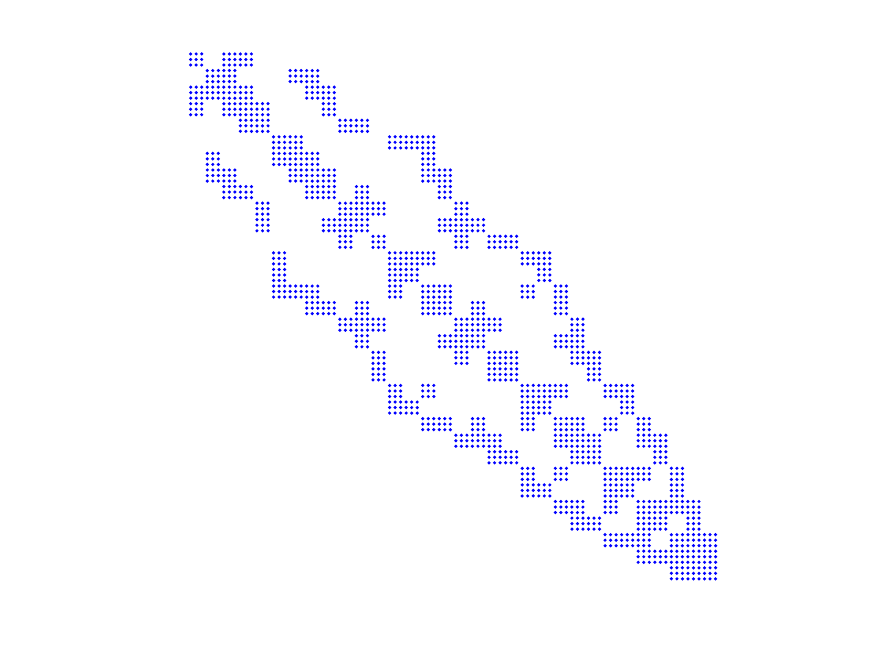}}
  \caption{Structure of the stiffness matrices $\mathbf{A}_{\rm DG}$,
    $\mathbf{A}_{\#}$ and $\mathbf{A}_{\calP}$ on Mesh~\ref{M3} with
    $\mathbb{P}_{1}$ elements.}\label{FigSMAT}
\end{figure}


\section{Numerical experiments}	
\label{sec:numerics}

We assess the performance of the finite element method \eqref{BP-DG}
on the domain $\Omega=(0,1)^2$, using $\gamma=1$
in \eqref{alpha_choice} throughout.

The nonlinear system from \eqref{FEMethod} is solved in two
steps. First, compute $w_h^0\in W_\calT$ from
\begin{equation}
  a_{\#}(w_{h}^{0},v_{h})=\langle f, v_{h}\rangle_{\Omega} \qquad \forall v_{h}\in W_\calT,
  \label{iter}
\end{equation}
and let $\mathbf{W}^0=(w_1^0,\ldots,w_{N_{\#}}^0)$ be the vector of
$W_\calT$-dofs. Using the algebraic form \eqref{R-FEM_Algebric}, for
$n=0,1,2,\ldots$ we update via a semi-smooth Newton method
(cf.~\cite{argyros1988newton})
\begin{equation}
  \mathbf{U}^{(n+1)}=\mathbf{Q}\mathbf{W}^{(n)}-J_{\mathbf{F}}(\mathbf{U}^{(n)})^{-1}\,\mathbf{F}(\mathbf{W}^{(n)}),
  \label{Newton}
\end{equation}
where $\mathbf{F}(\mathbf{W})=\mathbf{F}_\calP-\mathbf{A}_\calP\,\mathbf{Q}(\mathbf{W})^{+}-\mathbf{S}\,\mathbf{Q}(\mathbf{W})^{-}$ and
\begin{equation}
  J_{\mathbf{F}}(\mathbf{U})=J_{\mathbf{F}}(\mathbf{W})\mathbf{Q}^T
  =-\Big(\mathbf{A}_\calP\mathbf{Q}\,\mathrm{diag}(I^{1}(w_1),\ldots,I^{1}(w_{N_{\#}}))
  +\mathbf{S}\mathbf{Q}\,\mathrm{diag}(I^{2}(w_1),\ldots,I^{2}(w_{N_{\#}}))\Big)\mathbf{Q}^{T},
  \label{Jacobi}
\end{equation}
with indicator functions
\[
I^1(w_i)=\begin{cases}0,& w_i<0,\\ 1,& w_i\ge 0,\end{cases}
\qquad
I^2(w_i)=\begin{cases}0,& w_i\in[0,\kappa],\\ 1,& \text{otherwise},\end{cases}
\quad i=1,\ldots,N_{\#}.
\]
The relation \eqref{Jacobi} follows from
$J_{\mathbf{F}}(\mathbf{W})\mathbf{Q}^{T}\mathbf{Q}\mathbf{W}=J_{\mathbf{F}}(\mathbf{U})\mathbf{U}$. Iterations
terminate when
\begin{equation}
  \|\mathcal{E}(u_H^{n+1})-\mathcal{E}(u_H^{n})\|_{0,\Omega}\le 10^{-8},
  \label{stopping}
\end{equation}
where $\mathcal{E}(u_H^n)=w_h^n=\sum_{i=1}^{N_{\#}}W_i^n\phi_i$.

We report both asymptotic convergence and nonlinear solver
behaviour. Let $\#\calP$ be the number of polytopic elements. Given
consecutive meshes with $\#\calP_1$ and $\#\calP_2$ elements, the
estimated order of convergence (EOC) is
\begin{equation}
  \mathrm{EOC}
  = \frac{\log\!\left(\|e_{1}\|/\|e_{2}\|\right)}
         {\log\!\left(\sqrt{\#\calP_1}/\sqrt{\#\calP_2}\right)},
  \label{OOC}
\end{equation}
where $\|e_i\|$ is the error on the mesh with $\#\calP_i$ elements. In all plots and error calculations we display $\mathcal{E}^{+}(u_H)$.

\begin{example}[Convergence for a smooth solution]\label{Example1}
We take $\mu=1$ and
\[
\mathcal{D}=\epsilon\begin{bmatrix}100&\cos(x)\\ \cos(x)&1\end{bmatrix},\qquad \epsilon=10^{-6},
\]
with forcing $f$ chosen so that $u(x,y)=\sin(c\pi x)\sin(c\pi y)$ (hence $u\in[0,1]$ and $\kappa=1$).
\end{example}

We use $\mathbb{P}_1$, $\mathbb{P}_2$, and $\mathbb{P}_3$ elements on
the meshes of Figure~\ref{meshs}. Tables~\ref{Tab1}--\ref{Tab3} report
$L^2$, broken $H^1$, and DG errors of $u-\mathcal{E}^+(u_H)$, together
with $\|\mathcal{E}^-(u_H)\|_s$ and Newton iteration counts. The EOC
is computed via \eqref{OOC}.

\begin{table}[H]
  \centering
  \resizebox{\textwidth}{!}{ 
  \begin{tabular}{|c|c|c|c|c|c|c|c|c|c|}
  \hline
  $\#\calP$ & Itr. & $\|u-\mathcal{E}^{+}(u_{H}) \|_{0,\Omega}$ & EOC & $|u-\mathcal{E}^{+}(u_{H})|_{1,\Omega}$ & EOC & $\|u-\mathcal{E}^{+}(u_{H}) \|_{\rm DG}$ & EOC & $\|\mathcal{E}^{-}(u_{H})\|_{s}$ & EOC \\
  \hline\hline
   8 & 3 & 5.11e-1 & -- & 1.80e-2 & -- & 5.11e-1 & -- & 0 & -- \\
  16 & 3 & 5.22e-1 & -- & 1.79e-2 & -- & 5.22e-1 & -- & 0 & -- \\ 
  32 & 3 & 5.21e-1 & -- & 1.80e-2 & -- & 5.22e-1 & -- & 0 & -- \\ 
  64 & 3 & 3.98e-1 & -- & 1.73e-2 & 0.11 & 4.09e-1 & 0.70 & 0 & -- \\ 
  128 & 3 & 2.98e-1 & 0.83 & 1.48e-2 & 0.45 & 3.24e-1 & 0.67 & 0 & -- \\  
  256 & 10 & 1.73e-1 & 1.57 & 1.16e-2 & 0.70 & 2.15e-1 & 1.18 & 9.07e-3 & -- \\  
  512 & 7 & 8.76e-2 & 1.96 & 8.41e-3 & 0.93 & 1.27e-1 & 1.52 & 1.63e-2 & -- \\  
  1024 & 7 & 4.33e-2 & 2.03 & 5.97e-3 & 0.99 & 7.71e-2 & 1.44 & 1.30e-3 & 7.29 \\      
  \hline
  \end{tabular}}
  \caption{Convergence results for Example~\ref{Example1} using $\mathbb{P}_{1}$ and $c=8$.}
  \label{Tab1}
\end{table}

\begin{table}[H]
  \centering
  \resizebox{\textwidth}{!}{ 
  \begin{tabular}{|c|c|c|c|c|c|c|c|c|c|}
  \hline
  $\#\calP$ & Itr. & $\|u-\mathcal{E}^{+}(u_{H}) \|_{0,\Omega}$ & EOC & $|u-\mathcal{E}^{+}(u_{H})|_{1,\Omega}$ & EOC & $\|u-\mathcal{E}^{+}(u_{H}) \|_{\rm DG}$ & EOC & $\|\mathcal{E}^{-}(u_{H})\|_{s}$ & EOC \\
  \hline\hline
   8 & 3 & 5.17e-1 & -- & 1.81e-2 & -- & 5.17e-1 & -- & 0 & -- \\
  16 & 3 & 4.94e-1 & -- & 1.80e-2 & -- & 4.96e-1 & -- & 0 & -- \\ 
  32 & 3 & 4.20e-1 & 0.47 & 1.69e-2 & -- & 4.40e-1 & 0.35 & 0 & -- \\ 
  64 & 9 & 2.32e-1 & 1.71 & 1.28e-2 & 0.80 & 2.70e-1 & 1.41 & 3.73e-3 & -- \\ 
  128 & 10 & 1.12e-1 & 1.88 & 7.96e-3 & 1.37 & 1.35e-1 & 2.00 & 4.12e-2 & -- \\  
  256 & 13 & 4.41e-2 & 2.91 & 4.45e-3 & 1.68 & 5.93e-2 & 2.37 & 1.58e-2 & 2.26 \\  
  512 & 13 & 1.34e-2 & 3.43 & 2.33e-3 & 1.87 & 2.12e-2 & 2.96 & 3.12e-3 & 5.18 \\  
  1024 & 14 & 4.67e-3 & 3.04 & 1.12e-3 & 2.11 & 9.66e-3 & 2.27 & 1.63e-3 & 1.87 \\   
  \hline
  \end{tabular}}
  \caption{Convergence results for Example~\ref{Example1} using $\mathbb{P}_{2}$ and $c=8$.}
  \label{Tab2}
\end{table}

\begin{table}[H]
  \centering
  \resizebox{\textwidth}{!}{ 
  \begin{tabular}{|c|c|c|c|c|c|c|c|c|c|}
  \hline
  $\#\calP$ & Itr. & $\|u-\mathcal{E}^{+}(u_{H}) \|_{0,\Omega}$ & EOC & $|u-\mathcal{E}^{+}(u_{H})|_{1,\Omega}$ & EOC & $\|u-\mathcal{E}^{+}(u_{H}) \|_{\rm DG}$ & EOC & $\|\mathcal{E}^{-}(u_{H})\|_{s}$ & EOC \\
  \hline\hline
   8 & 3 & 5.12e-1 & -- & 1.82e-2 & -- & 5.18e-1 & -- & 0 & -- \\
  16 & 9 & 4.44e-1 & -- & 1.74e-2 & -- & 5.41e-1 & -- & 2.06e-2 & -- \\ 
  32 & 11 & 2.83e-1 & 1.30 & 1.32e-2 & 0.80 & 3.18e-1 & 1.53 & 4.54e-2 & -- \\ 
  64 & 13 & 1.04e-1 & 2.99 & 7.48e-3 & 1.64 & 1.45e-1 & 2.27 & 3.03e-2 & 1.17 \\ 
  128 & 14 & 2.72e-2 & 3.77 & 2.85e-3 & 2.78 & 3.47e-2 & 4.13 & 4.25e-3 & 5.67 \\  
  256 & 17 & 6.52e-3 & 4.12 & 1.04e-3 & 2.91 & 9.56e-3 & 3.71 & 2.73e-3 & 1.28 \\  
  512 & 13 & 1.73e-3 & 3.82 & 3.61e-4 & 3.05 & 2.72e-3 & 3.15 & 4.01e-4 & 5.53 \\  
  1024 & 14 & 4.31e-4 & 4.01 & 1.27e-4 & 3.01 & 8.00e-4 & 3.53 & 5.71e-5 & 5.62 \\   
  \hline
  \end{tabular}}
  \caption{Convergence results for Example~\ref{Example1} using $\mathbb{P}_{3}$ and $c=8$.}
  \label{Tab3}
\end{table}

The results show optimal rates for all degrees. Additionally,
refinement leads to a monotone decay of $\|\mathcal{E}^-(u_H)\|_s$, as
expected. Figure~\ref{fig:ex1plots} displays $\mathcal{E}^{+}(u_{H})$
for $c=8$ on a mesh with $\#\calP=512$ using $\mathbb{P}_{1}$
elements.

\begin{figure}[h!]
  \centering
  \subfloat[Top-down view.]{\includegraphics[width=0.35\textwidth]{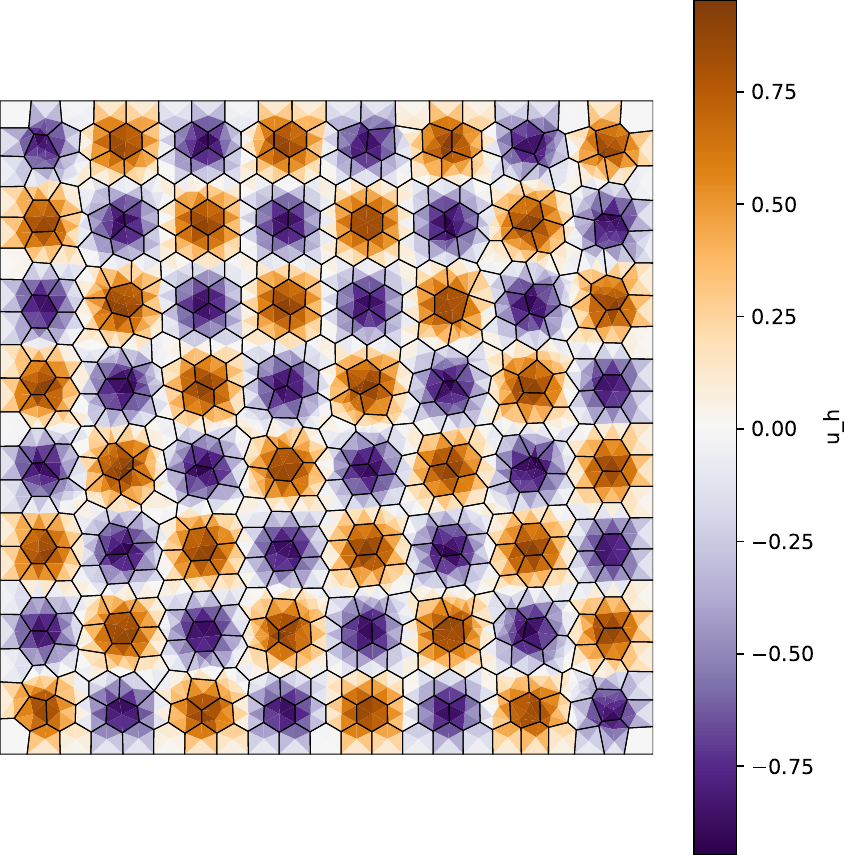}}
  \hspace{0.05\textwidth}
  \subfloat[Surface plot.]{\includegraphics[width=0.35\textwidth]{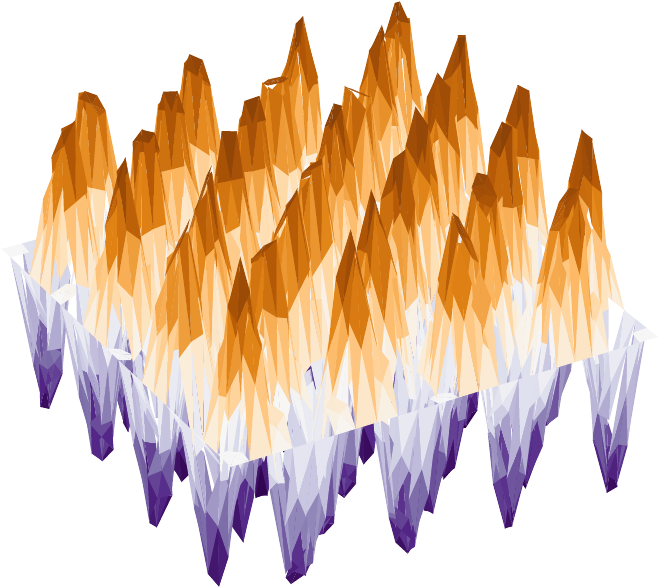}}
  \caption{Discrete solution $\mathcal{E}^{+}(u_{H})$ for Example~\ref{Example1} with $c=8$ on a mesh with $\#\calP=512$ using $\mathbb{P}_{1}$ elements.}
  \label{fig:ex1plots}
\end{figure}

\begin{example}[Resolution of boundary layers]\label{Example22}
We consider
\begin{align}
\left\{
\begin{array}{ll}
-\epsilon \Delta u + \mu u = 1, & \text{in } \Omega, \\
u = 0, & \text{on } \partial\Omega,
\end{array}
\right.
\label{Example2}
\end{align}
with $\epsilon\in[10^{-6},10^{-2}]$ and $\mu=1$. Since $u(x)\in[0,1]$, we take $\kappa=1$.
\end{example}

We report results for $\mathbb{P}_1$ and
$\mathbb{P}_2$. Figure~\ref{Fig6} shows $\mathcal{E}^{+}(u_{H})$ for
decreasing $\epsilon$; Table~\ref{Tab7} lists Newton iteration counts
for the stopping criterion \eqref{stopping}. As $\epsilon$ decreases,
boundary layers sharpen while the scheme remains stable and
convergent.

\begin{figure}[h!]
  \centering
  \subfloat[$\epsilon=10^{-2}$, top-down.]{\includegraphics[width=0.35\textwidth]{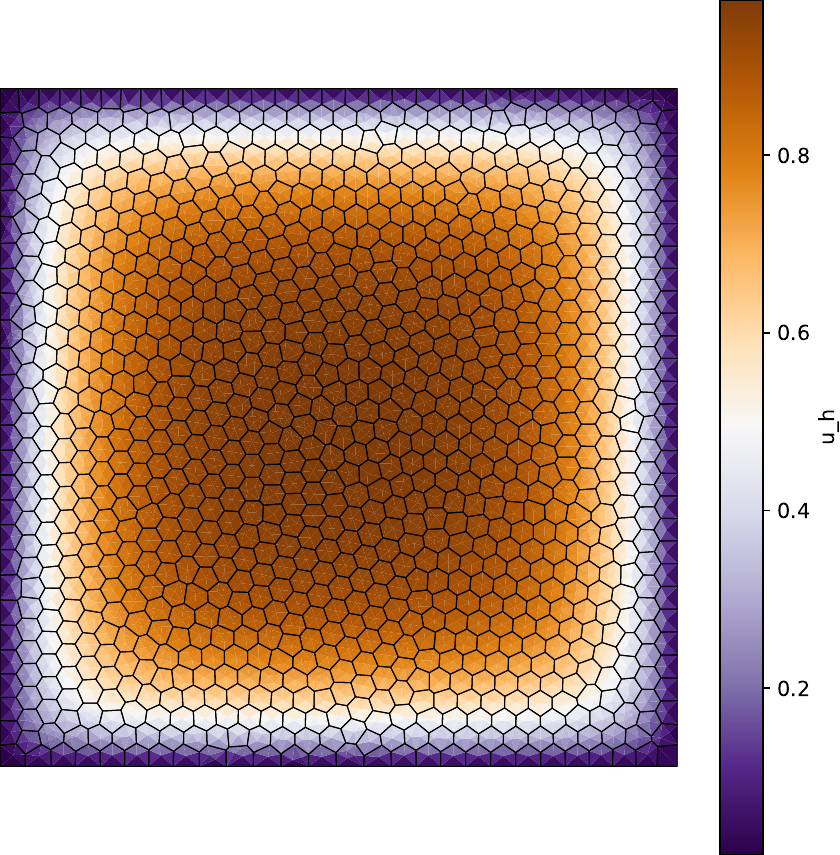}}
  \hspace{0.05\textwidth}
  \subfloat[$\epsilon=10^{-2}$, surface.]{\includegraphics[width=0.35\textwidth]{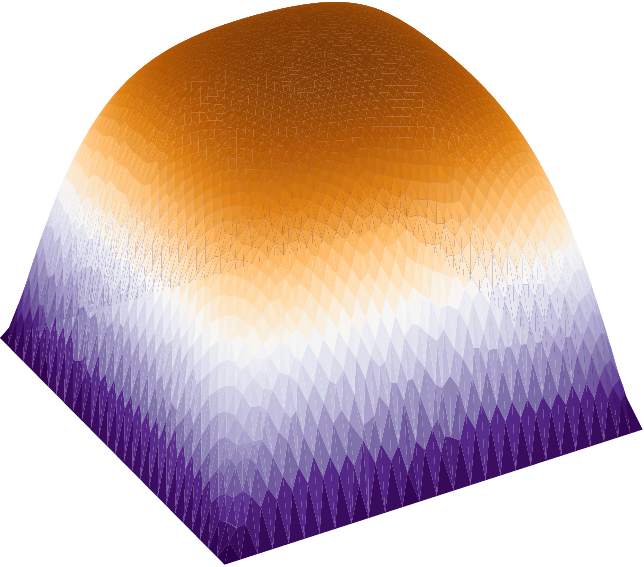}}\\[1ex]
  \subfloat[$\epsilon=10^{-4}$, top-down.]{\includegraphics[width=0.35\textwidth]{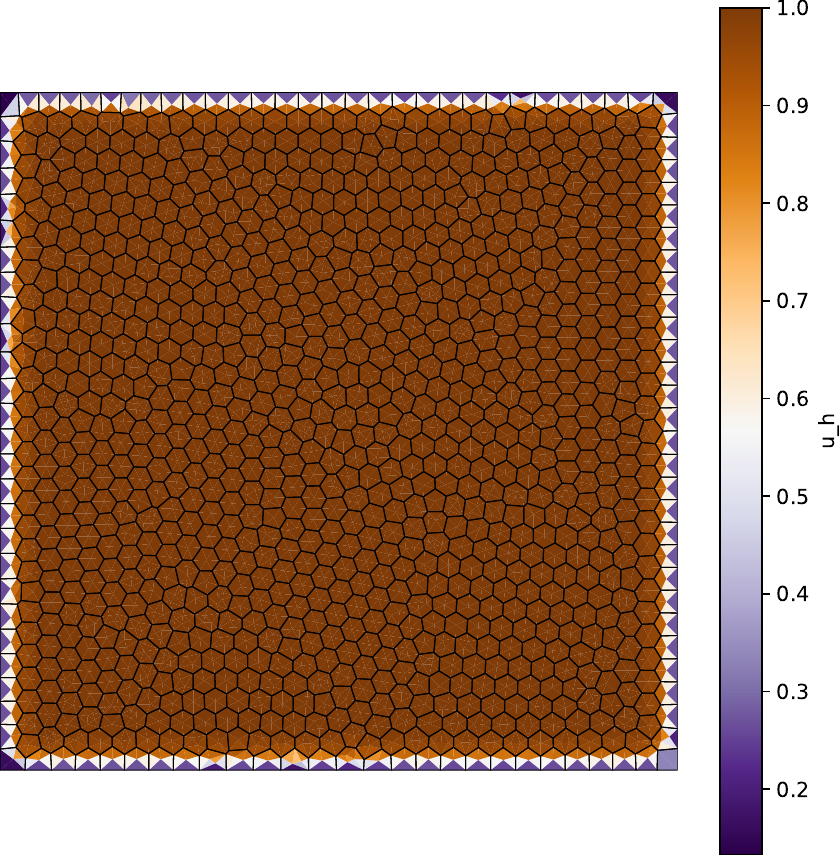}}
  \hspace{0.05\textwidth}
  \subfloat[$\epsilon=10^{-4}$, surface.]{\includegraphics[width=0.35\textwidth]{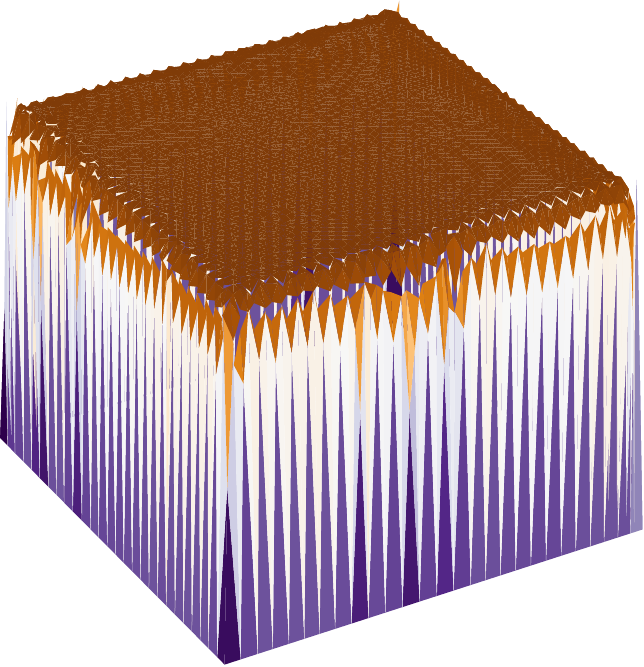}}\\[1ex]
  \subfloat[$\epsilon=10^{-6}$, top-down.]{\includegraphics[width=0.35\textwidth]{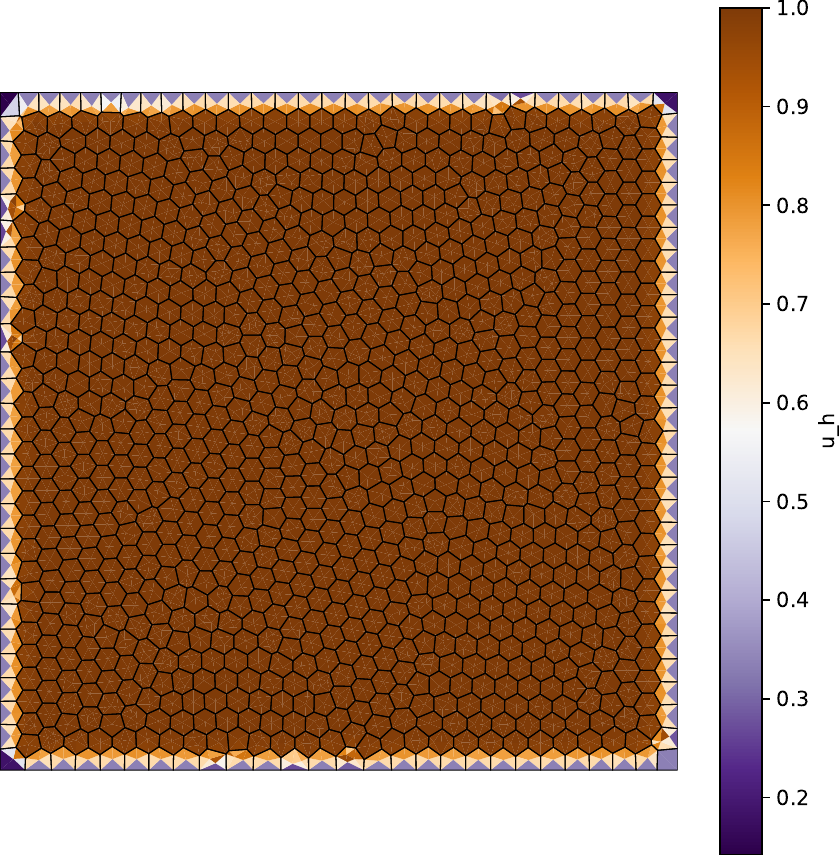}}
  \hspace{0.05\textwidth}
  \subfloat[$\epsilon=10^{-6}$, surface.]{\includegraphics[width=0.35\textwidth]{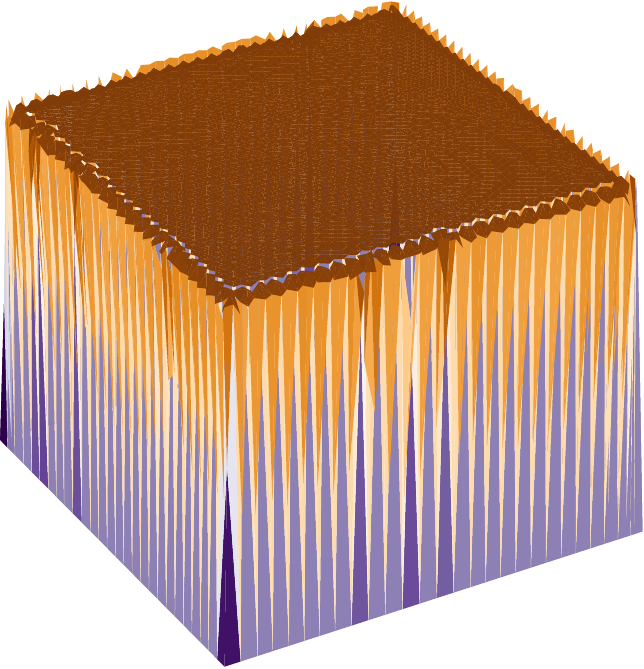}}
  \caption{Approximate solutions $\mathcal{E}^{+}(u_{H})$ for Example~\ref{Example22} using $\mathbb{P}_{2}$ on the finest mesh ($\#\calP=1024$).}
  \label{Fig6}
\end{figure}

\begin{table}[H]
  \centering
  \begin{tabular}{|c|c|c|c|c|c|c|}
    \hline
     & $\epsilon=10^{-7}$ & $10^{-6}$ & $10^{-5}$ & $10^{-4}$ & $10^{-3}$ & $10^{-2}$ \\
    \hline\hline
    $\mathbb{P}_{1}$ & 13 & 14 & 15 & 16 & 6 & 7 \\
    $\mathbb{P}_{2}$ & 14 & 14 & 14 & 20 & 7 & 8 \\
    \hline
  \end{tabular}
  \caption{Newton iterations for Example~\ref{Example22} using $\mathbb{P}_{1}$ and $\mathbb{P}_{2}$ on the finest mesh.}
  \label{Tab7}
\end{table}

\begin{example}[Solution with an interior layer]\label{Example33}
  We consider
  \begin{equation}
    \left\{
    \begin{aligned}
      -\epsilon \Delta u + u &= f && \text{in }\Omega,\\
      u &= 0 && \text{on } \partial\Omega,
    \end{aligned}
    \right.
    \label{Example3}
  \end{equation}
  with piecewise constant forcing
  \[
  f(x,y)=
  \begin{cases}
    \tfrac{1}{2}, & (x,y)\in \big[\tfrac{1}{4},\tfrac{3}{4}\big]^2,\\[2pt]
    1,            & \text{otherwise},
  \end{cases}
  \]
  and $\epsilon \in [10^{-6},10^{-2}]$. In this setting the solution develops an interior layer and attains
  a local minimum inside the domain. Since $u(x)\in [0,1]$, we set $\kappa=1$.
\end{example}

We report results for $\mathbb{P}_{1}$ and $\mathbb{P}_{2}$. Elevation plots on the finest polygonal mesh are shown in Figure~\ref{Fig8}, illustrating the sharpening of the interior layer as $\epsilon$ decreases. Figure~\ref{Fig11} compares cross–sections of the R-BP-FEM solution with those obtained by the standard DG method \eqref{wip} along the line $y=-x$. As $\epsilon\to 0$, the DG solution exhibits oscillations near the layer, whereas the R-BP-FEM removes these undershoots/overshoots while preserving the layer structure. Newton iteration counts for the stopping criterion \eqref{stopping} are summarised in Table~\ref{Tab8}.

\begin{figure}[h!]
  \centering
  \subfloat[$\epsilon=10^{-2}$, top-down.]
    {\includegraphics[width=0.35\textwidth]{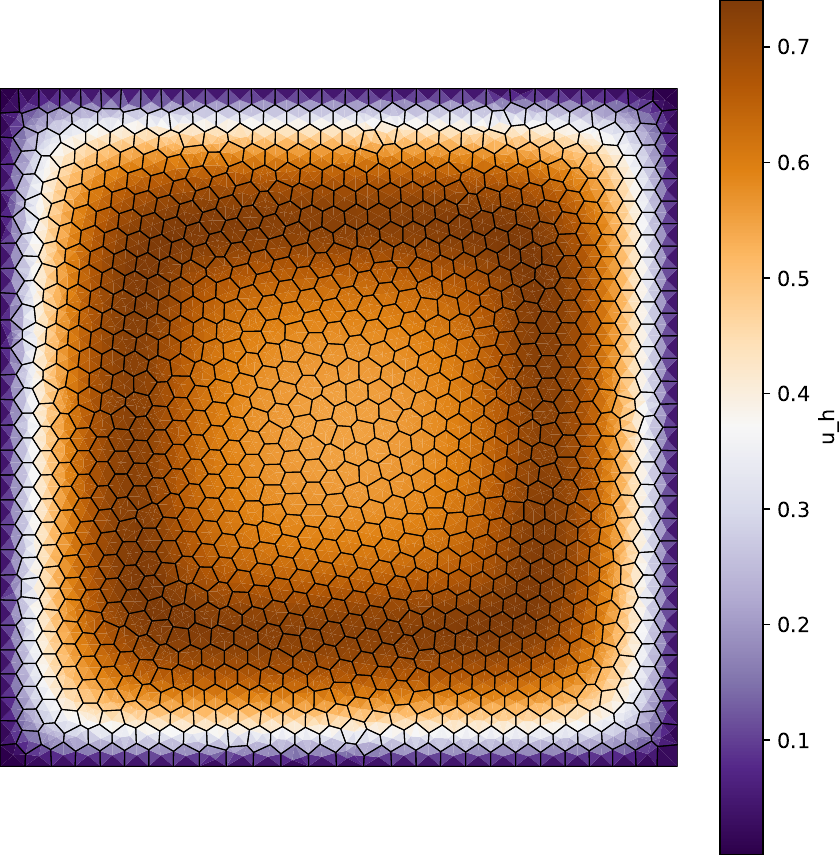}}
  \hspace{0.05\textwidth}
  \subfloat[$\epsilon=10^{-2}$, surface.]
    {\includegraphics[width=0.35\textwidth]{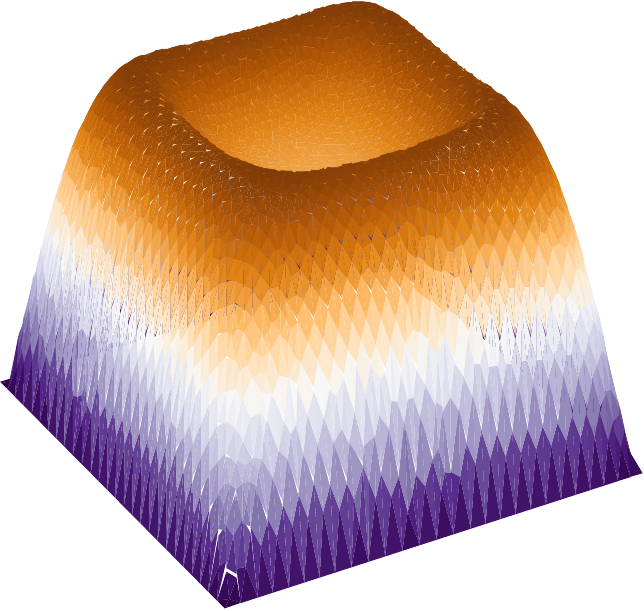}}\\[1ex]
  \subfloat[$\epsilon=10^{-4}$, top-down.]
    {\includegraphics[width=0.35\textwidth]{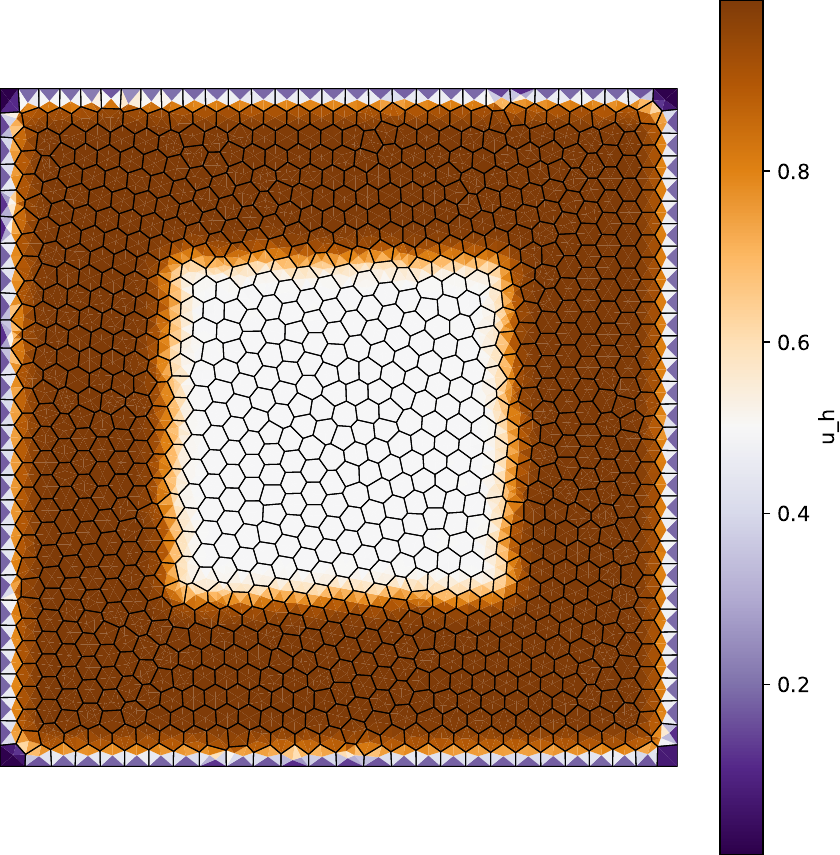}}
  \hspace{0.05\textwidth}
  \subfloat[$\epsilon=10^{-4}$, surface.]
    {\includegraphics[width=0.35\textwidth]{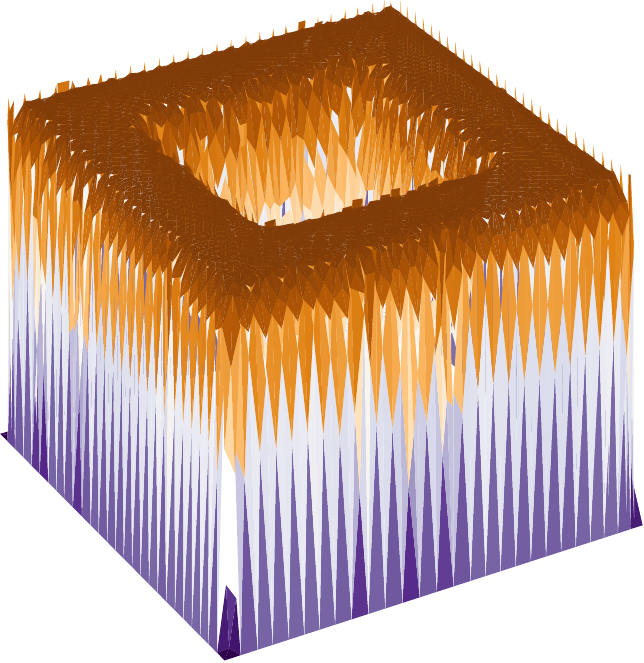}}\\[1ex]
  \subfloat[$\epsilon=10^{-6}$, top-down.]
    {\includegraphics[width=0.35\textwidth]{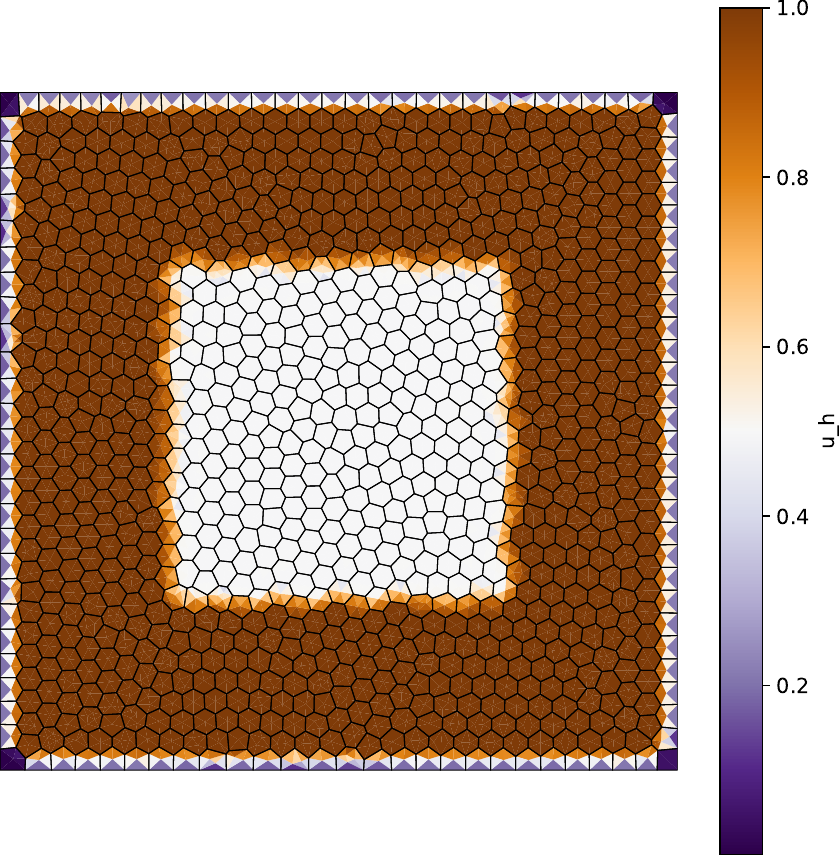}}
  \hspace{0.05\textwidth}
  \subfloat[$\epsilon=10^{-6}$, surface.]
    {\includegraphics[width=0.35\textwidth]{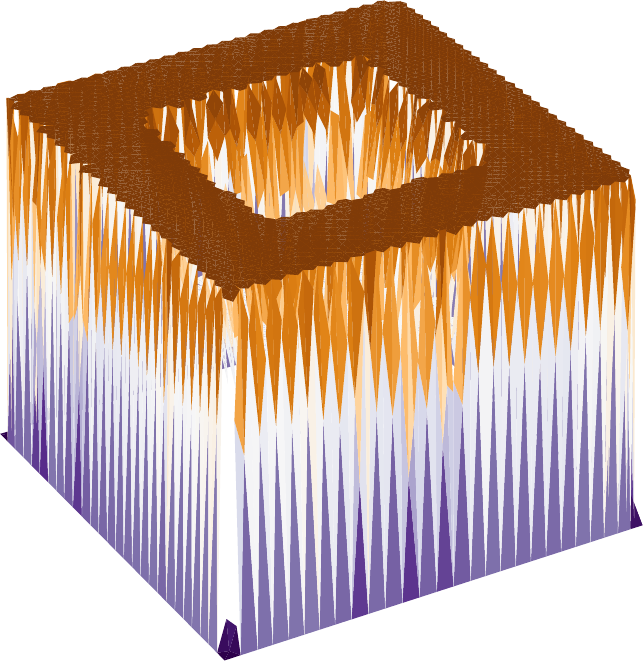}}
  \caption{Approximate solutions $\mathcal{E}^{+}(u_{H})$ for Example~\ref{Example33} using $\mathbb{P}_{1}$ on the finest mesh ($\#\mathcal{P}=1024$).}
  \label{Fig8}
\end{figure}

\begin{figure}[h!]
  \centering
  \subfloat[$\epsilon=10^{-4}$, $\mathbb{P}_{1}$.]{\includegraphics[width=0.45\textwidth]{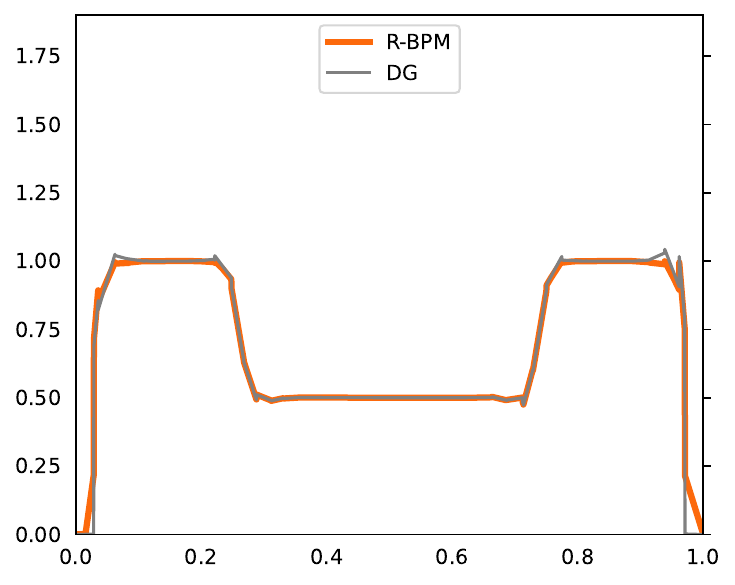}}
  \subfloat[$\epsilon=10^{-7}$, $\mathbb{P}_{1}$.]{\includegraphics[width=0.45\textwidth]{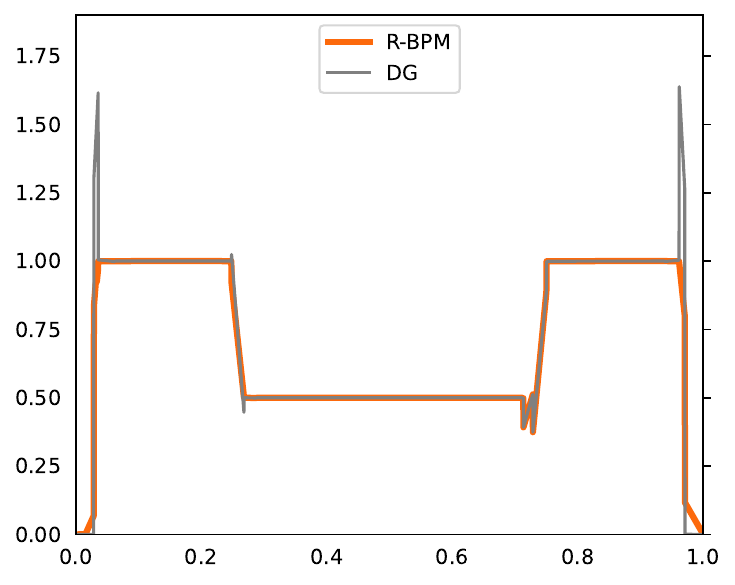}}\\
  \subfloat[$\epsilon=10^{-4}$, $\mathbb{P}_{2}$.]{\includegraphics[width=0.45\textwidth]{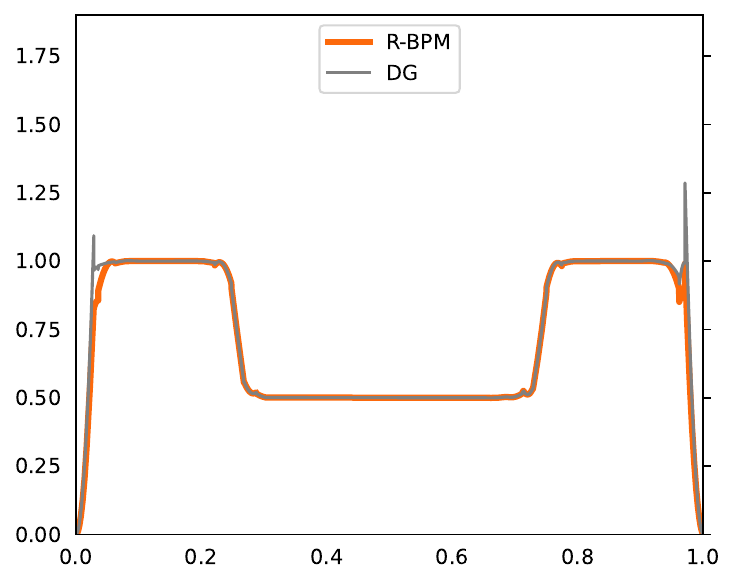}}
  \subfloat[$\epsilon=10^{-7}$, $\mathbb{P}_{2}$.]{\includegraphics[width=0.45\textwidth]{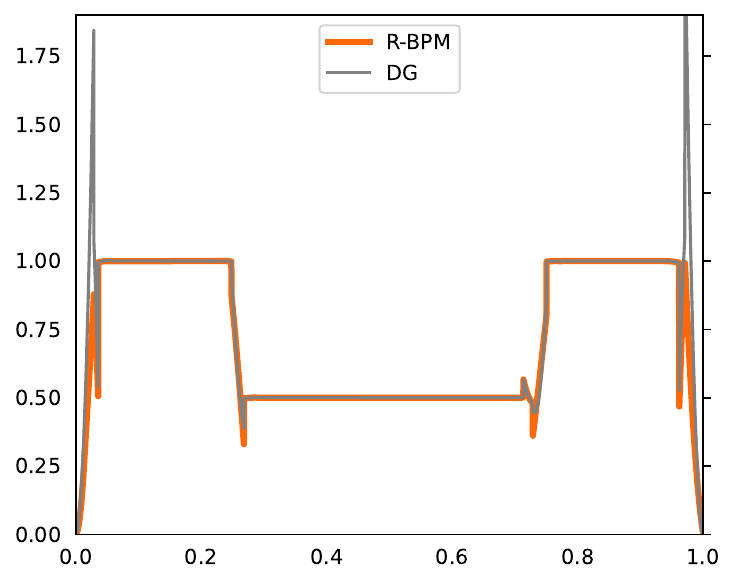}}
  \caption{Cross-sections along $y=-x$ for Example~\ref{Example33}: comparison of R-BP-FEM and DG \eqref{wip} for $\mathbb{P}_{1}$ and $\mathbb{P}_{2}$.}
  \label{Fig11}
\end{figure}

\begin{table}[H]
  \centering
  \begin{tabular}{|c|cccccc|}
    \hline
     & $10^{-7}$ & $10^{-6}$ & $10^{-5}$ & $10^{-4}$ & $10^{-3}$ & $10^{-2}$ \\
    \hline\hline
    $\mathbb{P}_{1}$\ Itr. & 13 & 14 & 14 & 17 & 13 & 7 \\
    $\mathbb{P}_{2}$\ Itr. &  9 &  7 & 21 & 14 & 13 & 9 \\
    \hline
  \end{tabular}
  \caption{Newton iterations for Example~\ref{Example33} on the finest mesh.}
  \label{Tab8}
\end{table}

\section{Conclusion}\label{sec:conclusion}

A bound-preserving discontinuous
Galerkin framework on polytopic meshes is introduced. The construction relies on a
projection-based formulation, in which a robust bilinear form is
combined with nonlinear stabilisation to ensure non-negativity of the
discrete solution without compromising the convergence behaviour of
the underlying DG scheme. The method accommodates general polygonal
meshes, allowing for flexibility in mesh generation and refinement
strategies, while maintaining compatibility with standard polynomial
approximation spaces. A salient feature of the approach is the ability to select at will the cardinality and the position of nodes where range bounds are enforced without affecting the overall complexity of the method.

The theoretical development has been complemented by an extensive
series of numerical experiments. These tests demonstrate optimal
convergence rates for smooth solutions, confirm robustness of the
scheme across a wide range of diffusion parameters and illustrate the
method's ability to resolve both boundary and interior layers without
spurious oscillations, or out-of-range over- or undershoots. Comparisons with a standard DG discretisation
confirm that the proposed method eliminates non-physical behaviour
in challenging singularly perturbed regimes.

\section*{Acknowledgements}

All authors gratefully acknowledge the financial support of The
Leverhulme Trust (grant number RPG-2021-238).  EHG was supported by
EPSRC grant number EP/W005840/2.  TP was supported by the EPSRC
programme grant EP/W026899/2.

\bibliographystyle{siam}
\bibliography{refs.bib}
\end{document}